\documentclass{article}

\usepackage{color}

\usepackage{amssymb,amsmath,amsthm,graphicx, color, amsfonts}
\usepackage{kpfonts} 
\usepackage[english]{babel}
\usepackage{authblk, relsize}
\usepackage{geometry}

\newtheorem{theorem}{Theorem}[section]

\newtheorem{lemma}[theorem]{Lemma}
\newtheorem{proposition}[theorem]{Proposition}
\newtheorem{corollary}[theorem]{Corollary}
\newtheorem{definition}[theorem]{Definition}
\newtheorem{remark}[theorem]{Remark}
\newtheorem{example}[theorem]{Example}

\newcommand{\N}{\mathbb{N}}
\newcommand{\R}{\mathbb{R}}
\newcommand{\btimes}{\mathlarger{\mathlarger{\mathlarger{\times}}}}
\newcommand{\bchi}{\mathlarger{\mathlarger{\chi}}}
\renewenvironment{proof}[1][.]{%
\bigskip\noindent{\bf Proof#1 }}{%
\hfill$\blacksquare$\bigskip}

\newcommand{\bR}{\mathbb R}

\begin{document}
\pagestyle{myheadings}
\title{Fuzzy Attractors Appearing from GIFZS}

\author[1]{Elismar R. Oliveira
\thanks{
Instituto de Matem\'atica e Estat\'istica - UFRGS\\
Av. Bento Gon\c calves, 9500\\
Porto Alegre - 91500 - RS -Brasil\\
Email: elismar.oliveira@ufrgs.br}
}
\affil[1]{Universidade Federal do Rio Grande do Sul\\
}
\author[2]{Filip Strobin
\thanks{
Instititute of Mathematics, {\L}\'od\'z University of Technology, W\'olcza\'nska 215, 90-924 {\L}\'od\'z}
}
\affil[2]{{\L}\'od\'z University of Technology\\
}
\date{\today}
\maketitle
\boldmath{ 

\noindent\rule{\textwidth}{0.4mm}
\begin{abstract}
Cabrelli, Forte, Molter and Vrscay in 1992 considered a {fuzzy} version of the theory of iterated function systems (IFSs in short) and their fractals
, which now is quite rich and important part of the fractals theory.

On the other hand, Miculescu and Mihail in 2008 introduced another generalization of the IFSs' theory - instead of selfmaps of a metric space $X$, they considered mappings defined on the finite Cartesian product $X^m$. 

In this paper we show that the \emph{fuzzyfication} ideas of Cabrelli et al. can be naturally adjusted to the case of mappings defined on finite Cartesian product. In particular, we define the notion of a generalized iterated fuzzy function system (GIFZS in short) and prove that it generates a unique fuzzy fractal set. We also study some basic properties of GIFZSs and their fractals, and consider the question whether our setting gives us some new fuzzy fractal sets.

\end{abstract}
\vspace {.8cm}
\noindent\rule{\textwidth}{0.4mm}

\tableofcontents

\vspace {.8cm}

\section*{Introduction}  
One of the milestones of the fractals theory is the Hutchinson-Barnsley theorem from early 80's (\cite{MR977274}, \cite{MR0625600}) which states that if $X$ is a complete matric space and $f_0,...,f_{n-1}:X\to X$ are Banach contractions (i.e., their Lipschitz constants $Lip(f_j)<1$), then there is a unique nonempty and compact set $A\subset X$ such that
$$
A=f_1(A)\cup...\cup f_n(A)
$$
Such sets $A$ are called \emph{fractals} or \emph{attractors}, and systems $(X,(f_0,...,f_{n-1}))$ of continuous (contractive) maps are called \emph{iterated function systems} (IFSs for short).
It turns out that many interesting abstract sets, for example the Cantor ternary set or the Sierpi\'{n}ski triangle, are such fractals. Also some ``natural" objects, like trees, clouds etc., have a fractal structure in a certain scale and the Hutchinson-Barnsley fractals theory give nice tools for modelling them.

One direction of studies of the Hutchinson-Barnsley (HB for short) theory origines with the question if we can look at fractal sets as certain \emph{fuzzy} sets. The idea of fuzzy sets, introduced by Zadeh \cite{MR0219427} in 1965, is that instead of saying that some element $x$ belongs or not to a set $A$, we can say that it belong to $A$ \emph{in a certain degree}, where this ``degree" is some number from $[0,1]$. Such a nice idea attracted many mathematicians and found many applications. In particular, Cabrelli et al. \cite{MR1192494} introduced the fuzzy version of HB theory. In this setting, fractals can be fuzzy sets, and a given IFS is somehow \emph{fuzzied} by additional family of maps.

Another direction of investigations of IFSs' theory was initiated by Miculescu and Mihail \cite{MR2415407} 2008 (see also \cite{MR2449187}, \cite{MR2595825} and  \cite{MR3011940}). Instead of selfmaps of a metric space $X$, they considered mappings defined on finite Cartesian product of $X$ (they called systems of such mappings as \emph{generalized IFSs}, GIFSs for short). It turns out that such systems of mappings generate sets which can be considered as fractal sets, and many parts of classical theory have natural counterparts in such a framework. What is also important, the class of GIFSs' fractals is essentially wider than the class of classical IFSs' fractals (see Strobin \cite{MR3263451}).

Our goal in this paper is to unify this two approaches. We will define a fuzzy version of GIFSs and prove that under natural contractive conditions, such fuzzy systems generates fuzzy fractal sets. Also, we will investigate some properties of such fractals, and deal with the question whether our ``unification" generates some essentially new fuzzy fractal sets. Since we want our paper to be self-contained, we will recall some basics of fuzzy sets theory, as well as fractals theory of Cabrelli et al. and Miculescu and Mihail.

The paper is organized as follows. In the next section we give some topological preliminaries and background of fuzzy sets. In Section 2 we recall the fuzzy IFS theory and GIFSs' theory. Section 3 is devoted to introducing a fuzzy version of GIFSs and their fractals. Finally, in the last section we will study some further properties of them.

\section{Preliminaries}
\subsection{Topological background}
For the proofs of the results presented here you can check the excellent handbook Aliprantis and  Border~\cite{MR2378491}. Let $X$ be a fixed topological space and $\overline{\mathbb{R}}=\bR\cup\{-\infty,\infty\}$ be the extended set of real numbers.

\begin{definition} \emph{We said that $u: X \to \overline{\mathbb{R}}$ is} upper semicontinuous\emph{ (usc) if, for each $c \in \mathbb{R}$ the set $u^{-1}([c, +\infty]):=\{ x \in X \; | \; u(x) \geq c\}$ is closed. Analogously,  we said that $u: X \to \overline{\mathbb{R}}$ is  }lower semicontinuous\emph{ (lsc) if $(-f)$ is usc.}
\end{definition}

The proof of the next two lemmas can be found in  \cite{MR2378491}, Lemma 2.41 and 2.42, p. 43.
\begin{lemma} Let $u_{t}: X \to \overline{\mathbb{R}}$, $t \in \mathcal{T}$ be a family of usc (resp. lsc) functions.  Then, the pointwise supremum (resp. infimum)
$$u(x):=\sup_{t \in \mathcal{T}}u_t(x)$$ is a usc  (resp. lsc)  function.
\end{lemma}

\begin{lemma} Let $u: X \to \overline{\mathbb{R}}$. $u$ is usc  (resp. lsc)  function if and only if for every net $(x_t)_{t\in\mathcal{T}}\subset X$ with $x_{t} \to x$ it follows that $\displaystyle \limsup_{x_t \to x} u(x_t) \leq u(x)$ (resp. $\displaystyle \liminf_{x_t \to t} u(x_t) \geq u(x)$). If $X$ is first countable (i.e., each point has a countable
neighborhood base, for example a metric space) the net $(x_t)_{t\in \mathcal{T}}$ can be replaced by a sequence.
\end{lemma}

The following result generalizes the Weierstrass theorem.
\begin{theorem} If $u: K\subseteq X \to \overline{\mathbb{R}}$ is  a usc  (resp. lsc)  function on the compact $K$, then $u$ attains its maximum (resp. minimum) value $\max_K u$ (resp. $\min_Ku$) and the set $$\mathrm{argmax}(u):= \{ x \in K \; | \; u(x)=\max_K u\}\;\;\;\;\;\mbox{(resp. }\mathrm{argmin}(u):=\{ x \in K \; | \; u(x)=\min_K u\}\mbox{)}$$ is nonempty and compact.
\end{theorem}
Now let us recall the Banach Fixed Point theorem.
\begin{theorem}
 \label{Banach Fixed Point theorem} Let $(A,d)$ be a complete metric space. Given a contraction $F: A  \to A$, there exists a unique $a \in A$ such that $F(a)= a.$
Moreover, for every $a_0 \in A$, the sequence $a_k, \; k \geq 0$ defined by
$$a_{k+1}:=F(a_{k }),$$
for all $k \in \mathbb{N}$, is convergent to $a$.
\end{theorem}
Finally, let us present the ``Collage Theorem" (the proof can be found in \cite{MR977274} of  Barnsley):
\begin{theorem}``Collage Theorem'' \label{Collage Theorem}  Let $(A,d)$ be a complete metric space and $T: A \to A$ be a Lipschitz contractive map, that is, $\mathrm{Lip(T)} <1$. Then for any $u \in A$ we have
$$d(u, u^*) \leq \frac{1}{1-\mathrm{Lip(T)}} \;d(u, T(u))$$
where $u^*$ is the unique fixed point of $T$.
\end{theorem}

\subsection{Basic definitions on Fuzzy Sets}
Let $X$ be a set.
\begin{definition} \emph{We say that $u$ is }a fuzzy subset of $X$\emph{ if $u: X \to [0,1]$. The family of fuzzy subsets of $X$ is denoted by $\mathcal{F}_{X}$, that is}
$$\mathcal{F}_{X}:=\{ u \; | \; u: X \to [0,1]\}.$$
\end{definition}

In this theory \emph{fuzzy set} means the that each point $x$ has a grade of membership  $0\leq u(x)\leq 1$ in the set $u$. Here, $u(x)=0$ indicates that $x$ is not in $u$ and $u(x)=0.4$ indicates that $x$ is a member of $u$ with membership degree $0.4$.

\begin{figure}[h!]
  \centering
  \includegraphics[width=2cm]{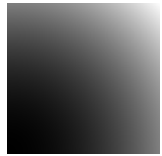}
  \caption{Representation of the fuzzy set $u(x,y)=1/2 (x^2 + y^2)$ in $X=[0,1]^2$ as a grey scale figure.  In this case $u(1,1)=1$=a white pixel and $u(0,0)=0$=a black pixel.}\label{example fuzzy set}
\end{figure}

\begin{definition}\label{grey level} \emph{Given $\alpha \in (0,1]$ and $u \in \mathcal{F}_{X}$, }the grey level \emph{or} $\alpha$-cut of $u$\emph{ is the crisp set
$$[u]^{\alpha}:=\{x \in X \; | \; u(x)\geq \alpha \},$$
that is, the set of points where the grey level exceeds the threshold value $\alpha$.
For $\alpha=0$ we define
$$[u]^{0}:={\operatorname{supp}(u):=} \overline{\bigcup\{[u]^{\alpha} \; | \;  \alpha >0\}}=\overline{\{x\in X:u(x)>0\}}$$}
\end{definition}
\begin{remark}\label{grey level'}
\emph{Observe that
$$[u]^{0}= \overline{\bigcup_{n=0}^{\infty}[u]^{\alpha_{n}}},$$
for every sequence $(\alpha_n)$ of positive reals with $\alpha_n \searrow 0$. It happens because the sequence of sets $[u]^{\alpha_{n}}$ is nondecreasing since  $[u]^{\alpha_{n}} \subseteq [u]^{\alpha_{n+1}} $ when $\alpha_{n} > \alpha_{n+1}$. In particular, the set $\bigcup_{n=0}^{\infty}[u]^{\alpha_{n}}$ (which is F-$\sigma$ provided $u$ is usc) is dense in $[u]^{0}$.}
\end{remark}

\begin{definition} \emph{A fuzzy set $u \in \mathcal{F}_{X}$ is\\
a) }a crisp set\emph{, if $u(x) \in \{0, 1\}$ for every $x\in X$. We identify it with the classic subset $U=\{ x \in X \; |\; u(x)=1 \}$. In this case, $u$ is the indicator function of $U$: $u(x)=\bchi_{U}(x)$;\\
b) }normal, \emph{if there is $x \in X$ such that $u(x)=1$;\\
c) }the universe, \emph{if $u(x)\equiv 1 = \bchi_{X}(x) $;\\
d) }empty, \emph{if $u(x)\equiv 0 = \bchi_{\varnothing}(x) $.\\
If additionally $X$ is a topological space, then we say that $u$ is\\
e) }compactly supported\emph{ if $[u]^0$ is compact.}
\end{definition}

Actually, the family of subsets of $X$, denoted by $2^X$, can be identified as a subset of $\mathcal{F}_X$, using the injective map $\chi: 2^X \to  \mathcal{F}_{X}$ defined by
$\bchi(B)=\bchi_{B}(x),$
for any $B \in 2^X$.

Given $f,g : X \to \mathbb{R}$, is usual to denote
$(f \vee g)(x) := \max \{f(x) , g(x)\} \text{ and } (f \wedge g)(x) := \min \{f(x) , g(x)\}.$ It shows how to define the fuzzy algebra of subsets.

\begin{definition} \emph{Given $u, v \in \mathcal{F}_{X}$ we define:\\
a) $u \cup v := u \vee v \in  \mathcal{F}_{X}$ and $u \cap v := u \wedge v \in  \mathcal{F}_{X}$,  }the union\emph{ and }the intersection of $u$ and $v$\emph{, respectively;\\
b) $u' := 1 - u \in  \mathcal{F}_{X}$,  }the complement\emph{ of $u$;\\
c) $u \subseteq v$ if $u(x) \leq v(x), \; \forall x \in X$,} the inequality relation.
\end{definition}
\begin{remark}\emph{It is well known that
the basic operations $\cup$ and $\cap$ with fuzzy sets:\\
a) are associative and distributive;\\
b) satisfies De Morgan's Laws
$(u \cap v)'=u' \cup v'\text{ and } (u \cup v)'=u' \cap v'.$
}
\end{remark}

We also have other algebraic operations.
\begin{definition} \emph{Given $u, v \in \mathcal{F}_{X}$ we define:\\
a) $u  v := u(x)v(x) \subseteq u \cap v \in \mathcal{F}_{X}$, }the algebraic product\emph{;\\
b) $u + v := \min\{u(x)+ v(x),1\} \in \mathcal{F}_{X}$, }the algebraic sum\emph{;\\
c) $|u - v |:= |u(x)-v(x)| \in \mathcal{F}_{X}$, }the absolute difference\emph{.\\
d) $ t u + (1-t)v \in \mathcal{F}_{X}$ for $t \in [0,1]$, }the convex combination of $u$ and $v$;\\
e) $(u,v)_{\Delta}:= \Delta(x) u + \Delta'(x) v$, the $\Delta$-convex combination of $u$ and $v$\emph{, where $\Delta \in \mathcal{F}_{X}$.}
\end{definition}

Fuzzy sets can be induced by maps. In his pioneering work in the 1965 Zadeh~\cite{MR0219427}, p. 346,  introduced what we call \emph{The Extension Principle}, that is a kind of pushforward map between fuzzy subsets.

\begin{definition}(Zadeh's Extension Principle) \emph{Given a map $T: X \to Y$, $ u \in \mathcal{F}_{X}$ and $v \in \mathcal{F}_{Y}$, we define new fuzzy sets $T(u)\in \mathcal{F}_{Y}$ and $T^{-1}(v)\in \mathcal{F}_{X}$ as follows\\
a) $T(u): Y \to [0,1]$ is given by
$$T(u)(y) :=
\left\{
  \begin{array}{ll}
    \sup_{T(x)=y} u(x), & if \; y \in T(X); \\
    0, & \mathrm{otherwise.}
  \end{array}
\right.
$$
b) $T^{-1}(v): X \to [0,1]$ is given by
$$T^{-1}(v)(x) := v(T(x)).$$}
\end{definition}
\begin{remark}\emph{
In He et al.~\cite{MR1729417}, Definitions 2.1, 2.2 and 2.3, we can found some alternative ways to define the extension principle for real~\footnote{After Zadeh's works the Fuzzy Set theory has been extended in several ways. In a wider sense, given $X$ a set and $R$ being usually some topological space, we define a $R$-valued fuzzy set as a function $u: X \to R_{g}$, where $R_g$, called the range, is a compact subset of $R$. In this paper, $X$ is a complete (or even compact) metric space, $R=\mathbb{R}$ and $R_g=[0,1]$. Measure-valued, set-valued, interval-valued and type-k ($R_g$ is a hypercube in $\mathbb{R}^k$) fuzzy sets are frequently considered in applications.} valued fuzzy sets, }The Supremum Extension Principle\emph{ (Zadeh's Extension Principle), }The Minimum Extension Principle\emph{ (the same as Zadeh's Extension Principle, replacing by minimization on the preimages) and }The Average Extension Principle\emph{, respectively. If $T^{-1}(y)$ is always finite we can define the Average Extension Principle
$$\tilde{T}(u)(y) :=
\left\{
  \begin{array}{ll}
    \frac{1}{\sharp\{ x|T(x)=y\}}\sum_{T(x)=y} u(x), & if \; y \in T(X); \\
    0, & \mathrm{otherwise.}
  \end{array}
\right.
$$
Obviously $0\leq \tilde{T}(u) \leq T(u)\leq 1$.}
\end{remark}

\begin{remark}\label{Crisp Set Extension to Fuzzy Set}\emph{
It may be instructive to see how $T$ works for a crisp set. If  $u(x)=\bchi_{B}(x)\in \mathcal{F}_{X},$ for some $B \in 2^X$, we get\\
$$T(u)(y) = \sup_{T(x)=y} u(x)= \sup_{T(x)=y} \bchi_{B}(x)=
\left\{
 \begin{array}{ll}
 1, & if \; y \in T(B) \\
0, & if y \; \not\in T(B)
\end{array}
\right. = \bchi_{T(B)}(y).
$$
That is, $T(\bchi_{B})= \bchi_{T(B)}$.\\
Similarly, $T^{-1}(\bchi_{C})=\bchi_{T^{-1}(C)}$ for all $C\in 2^Y$.}
\end{remark}

\begin{proposition}\label{compose usc plus cont}
Assume that $X$ and $Y$ be metric spaces and $f:X \to Y$ a continuous map. Given $u \in \mathcal{F}_{X}$ we have\\
a) If $u$ is normal then $f(u)$ is normal;\\
b) If $u$ is usc and compactly supported, then  $f(u)$ is usc and compactly supported.
\end{proposition}
\begin{proof}
a) Suppose that $u$ is normal, that is, there exists $a \in X$ such that $u(a)=1$. Let us evaluate $f(u)$ in $b=f(a)$
\[1 \geq f(u)(b) =
\left\{
  \begin{array}{ll}
    \sup_{f(x)=b} u(x), & if \; b \in f(A) \\
    0, & otherwise
  \end{array}
\right.
 \geq u(a) =1,\]
so $f(u)(b) =1$.\\
b) Assume that $u$ usc and compactly supported. We need to prove that $f(u)$ is so. At first, we prove that $f(u)^{-1}([c, +\infty])$ is closed for any $c \in \overline{\mathbb{R}}$. Since $0 \leq f(u) \leq 1$ we have
$$ f(u)^{-1}([c, +\infty]) =
  \begin{cases}
     \varnothing, & if \; 1<c\\
     [f(u)]^{c},  & if \; 0< c \leq 1\\
     B,  & if \; c\leq 0
  \end{cases}
$$
Since $\varnothing$ and $Y$ are closed, remains to prove that $[f(u)]^{c}$ for $0< c \leq 1$.

Let $(b_n) \subset [f(u)]^{c}$ and $b$ be its limit. We claim that $b \in [f(u)]^{c}$.  Since $f(u)(b_n) \geq c >0$ thus $f^{-1}(b_n) \neq \varnothing$. Now fix $\varepsilon\in(0,c)$ and for any $n\in\N$, let $a_n \in X$ be such that $u(a_n) \geq c-\varepsilon>0$ and $f(a_n)=b_n$. Since $u$ is usc and compactly supported, and $(a_n)\subset [u]^{c-\varepsilon}$, there is a subsequence $(a_{n_k})$ such that  $a_{n_{k}} \to a$ for some $a\in[u]^{c-\varepsilon}$. Also, by continuity of $f$, we have $b=f(a)$.
Thus
$$f(u)(b) =    \sup_{f(x)=b} u(x) \geq u(a) \geq c-\varepsilon$$
Since $\varepsilon$ was taken arbitrarily, we have $f(u)(b)\geq c$, which means that
$b \in [f(u)]^{c}$. So $[f(u)]^{c}$ is closed.\\
To see that $f(u)$ is compactly supported, observe that for any $y\in Y$ with $f(u)(y)>0$, there exists $x\in X$ such that $u(x)>0$ and $f(x)=y$. Hence
$$\{y\in Y:f(u)(y)>0\}\subset f(\{x\in X:u(x)>0\})\subset f (\overline{\{x\in X:u(x)>0\}})
$$
since the last set is compact (as $f$ is continuous and $u$ is compactly supported), we get that also $$\overline{\{y\in Y:f(u)(y)>0\}}$$ is compact. Hence $f(u)$ is compactly supported.

\end{proof}

\section{IFS fuzzyfication and generalized IFSs}

\textbf{To avoid any confusions we will reserve the calligraphic  $\mathcal{R}$ exclusively for IFS and the calligraphic  $\mathcal{S}$ will be reserved exclusively for generalized IFS that we will consider after.}

Now we turn our attention to the discrete dynamics of fuzzy sets.
On one hand the IFS offers for each iterate $\phi_j$ one of the possible positions $\phi_j(x)$ from an initial point $x$, that is the dynamics. On the other hand the IFZS offers one of the possibility functions $u_j=\phi_j u$ from an initial possibility function $u(x)$ that is, now we have a dynamics of possibility functions where $u(x)$ represents the possibility of a ``particle" be in the site $x \in X$ and $u_j(x)$ represents the possibility of the iteration of a ``particle" be in the site $\phi_j(x)$. The analogy is that in the classical mechanics the dynamics is given by a differential equation that defines the position $x$ and, when we make a quantification we deal with the evolution of probability distribution of the position via a unitary operator. We are going to develop this ideas using the notion of fuzzy sets. From the fuzzy point of view the possibility of a ``particle" be in the site $x \in X$, a metric space, is some number $u(x) \in [0,1]$, so the iterations generated by an IFS of the function $u$ must go through an appropriated fuzzy operator producing a new fuzzy set. We advise that it is not a probabilistic theory.

\subsection{IFS fuzzification}

The word fuzzification has several uses in the literature. Here, \emph{fuzzification} means to consider the analogous for fuzzy sets of the Hutchinson-Barnsley Theory for IFS acting on classical sets. The main ideas were developed in Cabrelli et al. \cite{MR1192494}. We repeat some results here to help the reader with no familiarity with this theory. Note that we extended a bit some of them in view of our study of GIFZSs in the next section.

We assume here that $(X,d)$ is a given metric space. Recall that the family of (real valued) fuzzy subsets of $X$ is defined by
$$\mathcal{F}_{X}:=\left\{ u \; | \; u: X \to [0,1] \text{ is a function}\right\},$$
and if $u\in\mathcal{F}_X$ and $\alpha\in(0,1]$, then $[u]^\alpha:=\{x\in X:u(x)\geq \alpha\}$ and also $[u]^0:=\overline{\{x\in X:u(x)>0\}}$.

To make this theory works we need to restrict $\mathcal{F}_{X}$ to a smaller family,
$$\mathcal{F}_{X}^* := \{ u \in \mathcal{F}_{X} \; | u \text{ is normal, usc and compactly supported}\} \; $$

\begin{proposition}
 If $u \in \mathcal{F}_{X}^*$ then for every $\alpha\in[0,1]$, the $\alpha$-cut set of $u$ is nonempty and compact.
\end{proposition}
\begin{proof}
The set $[u]^0$ is compact since $u$ is compactly supported. Now
  let $\alpha \in (0,1]$. Then we have $[u]^{\alpha} \neq \varnothing$ because $u$ is normal. Also, $[u]^\alpha$ is closed because $u$ is usc. Hence it is a closed subset of a compact set $[u]^0$, so also compact.
\end{proof}

The topology on $\mathcal{F}_{X}^*$ is defined by the Hausdorff distance between the $\alpha$-cuts.
We recall that in the set $\mathcal{K}^*(X)$ of nonempty and compact crisp subsets of $X$, the Hausdorff distance is defined by
$$h(A,B):= \max\left( \sup_{x \in A}\inf_{y \in B}d(x,y), \sup_{y \in B}\inf_{x \in A }d(x,y)\right).$$

Equivalently, if we define $A_{\varepsilon}:=\{ x \in X\; | \; d(x, A)\leq \varepsilon\}$, where $d(x,A):=\inf_{y\in A}d(x,y)$, then we get
$$h(A,B)= \inf\{\varepsilon >0\; | \;A \subseteq B_{\varepsilon},  B \subseteq A_{\varepsilon}\}.$$

Since $\mathcal{K}^*(X)$ contains all the $\alpha$-cuts, we can define a distance $d_\infty$ in $\mathcal{F}_{X}^*$ by
$$d_\infty (u,v) := \sup_{\alpha \in [0,1]} h([u]^{\alpha},[v]^{\alpha}),$$
for $u, v \in \mathcal{F}_{X}^*$. It is known that $d_\infty$ is a metric (see Diamond and Kloeden~\cite{MR1337027}), which is complete provided $X$ is compact (see a.e., Cabrelli et al.~\cite{MR1192494}). We will extend this result a bit (probably it is known, but we did not find a reference).

\begin{theorem} \label{Fuzzy Space is Complete} The function  $d_\infty : \mathcal{F}_{X}^* \times \mathcal{F}_{X}^* \to \mathbb{R}$ is a metric and $(\mathcal{F}_{X}^* , d_\infty)$ is a complete metric space provided $(X,d)$ is complete.
\end{theorem}
\begin{proof}
Let $(u_n)\subset \mathcal{F}_X^*$ be a Cauchy sequence. By definition, this means in particular that the sequence $([u_n]^0)$ is Cauchy in $\mathcal{K}(X)$, so, by completeness of $\mathcal{K}^*(X)$, it is convergent. This implies that $X':=\overline{\bigcup_{n\in\N}[u_n]^0}$ (equal to the union of all $[u_n]^0$ and the limit) is compact. Now for every $n\in\N$, let $u_n':=u_n\vert X'$ be the restriction of $u_n$ to $X'$. It is easy to see that $(u_n')$ is a Cauchy sequence in $\mathcal{F}_{X'}^*$. Since $X'$ is compact, $\mathcal{F}^*_{X'}$ is complete, so $u_n'\to u'$ for some $u'\in \mathcal{F}_{X'}^*$(by mentioned result from \cite{MR1192494}). Then $u_n\to u$ in $\mathcal{F}_X^*$, where $u$ is the natural extension of $u'$ to the whole $X$.
\end{proof}

Now we show that the definition of $d_\infty$ can be simplified a bit. We will use the following technical result from \cite{MR1192494} (Lemma A.1.)
\begin{lemma} \label{convergence of compact subsets}
If $(A_n)$ is a sequence of sets in $\mathcal{K}^*(X)$ such that ${A}_{n} \subseteq {A}_{n+1}$ for all $n\geq0$ and $A:=\overline{\bigcup_{n\geq 0}{A}_n}\in \mathcal{K}(X)$, then ${A}_{n} \to {A}$ with respect to the Hausdorff distance, that is, $h({A}_{n},{A}) \to 0$.
\end{lemma}

\begin{corollary}\label{f2}For any $u\in\mathcal{F}^*_X$ and a decreasing sequence $(\alpha_n)\subset (0,1]$ with $\alpha_n\to0$, we have $\displaystyle [u]^0=\lim_{n\to\infty}[u]^{\alpha_n}$ in the Hausdorff distance. In particular, $$d_\infty(u,v)=\sup_{\alpha\in(0,1]}h([u]^\alpha,[v]^\alpha)$$
\end{corollary}
\begin{proof}
The first part follows from Lemma~\ref{convergence of compact subsets} and Remark \ref{grey level'}. The second follows from the first one since
$$
h([u]^0,[v]^0)=h(\lim_{n\to\infty}[u]^{1/n},\lim_{n\to\infty}[v]^{1/n})=\lim_{n\to\infty}h([u]^{1/n},[v]^{1/n}).
$$
\end{proof}

\begin{definition}\emph{
A  \emph{grey level map} is a nonzero function $\rho: [0,1]\to [0,1]$. We said that a grey level map satisfy} ndrc condition \emph{ or is} an ndrc map\emph{, if\\
a) $\rho$ is nondecreasing;\\
b) $\rho$ is right continuous.}
\end{definition}


\begin{proposition} \label{alpha cut compose with ndrc rho}
Assume that $\rho$ is an ndrc map and $u \in \mathcal{F}_{X}$ is usc.\\
a) The map $\beta:[0,\rho(1)] \to [0,1]$, given by
$$\beta(\alpha):=\inf \{ t \; | \; \rho(t) \geq \alpha\}$$
is  well defined, nondecreasing and $\rho(\beta(\alpha))\geq \alpha$.\\
b) If $\alpha\in(0,1]$, then
$$
[\rho(u)]^\alpha=\left\{\begin{array}{ccc}
X & \mathrm{if} & \alpha\leq\rho(1) \;\mathrm{and}\; \beta(\alpha)=0 \\
{[}u{]}^{\beta(\alpha)}  & \mathrm{if} & \alpha\leq\rho(1)\;\mathrm{and}\;\beta(\alpha)>0 \\
\emptyset & \mathrm{if} &\alpha>\rho(1)
\end{array}\right.
$$
c) If $r_+:=\inf\{t:\rho(t)>0\}$, then
$$
[\rho(u)]^0=\left\{\begin{array}{ccc}
X & \mathrm{if} & \rho(0)>0\\
\overline{\bigcup_{\alpha>r_+}[u]^{\alpha}}  & \mathrm{if} & \rho(0)=0\;\mathrm{and}\; \rho(r_+)=0 \\
{[}u{]}^{r_+} & \mathrm{if} & \rho(0)=0\;\mathrm{and}\rho(r_+)>0
\end{array}\right.
$$
\end{proposition}
\begin{proof}
a) We know that if $a \in \{ t \; | \; \rho(t) \geq \alpha\}$ then $[a,1] \subset  \{ t \; | \; \rho(t) \geq \alpha\}$ because $\rho$ is nondecreasing. So there is an unique $\beta=\inf \{ t \; | \; \rho(t) \geq \alpha\}$, in particular $\beta \in \{ t \; | \; \rho(t) \geq \alpha\}$ because $\rho$ is right continuous. Take $\delta>0$ such that $\alpha<  \alpha+\delta \leq \rho(1)$ then $$\{ t \; | \; \rho(t) \geq \alpha+\delta\} \subset  \{ t \; | \; \rho(t) \geq \alpha\}$$ thus $\inf \{ t \; | \; \rho(t) \geq \alpha+\delta\} \geq  \inf  \{ t \; | \; \rho(t) \geq \alpha\}$ or $\beta(\alpha+\delta) \geq \beta(\alpha)$.

b) If $0<\alpha \leq \rho(1)$ and $\beta(\alpha)>0$, then to show that  $[\rho(u)]^{\alpha}= [u]^{\beta(\alpha)}$ we take $ x \in [\rho(u)]^{\alpha}$. Then  $\rho(u (x)) \geq \alpha$ that is $u (x) \in \{t\; | \; \rho(t) \geq \alpha \}$ thus $u (x) \geq \beta(\alpha)$. So $x \in [u]^{\beta(\alpha)}$.  Reciprocally, if $x \in [u]^{\beta(\alpha)}$ we get $u (x) \geq \beta(\alpha)$ and applying $\rho$ we get $\rho(u (x)) \geq \rho(\beta(\alpha)) \geq \alpha$ thus $ x \in [\rho(u)]^{\alpha}$.\\
If $\beta(\alpha)=0$, then by the right continuity of $\rho$ we have that $\rho(t)\geq\alpha$ for all $t\geq 0$, so for every $x\in X$, $\rho(u(x))\geq \alpha$.\\
Finally, since for every $x\in X$, $\rho(u(x))\leq\rho(1)$, we get that $[\rho(u)]^\alpha=\emptyset$ for $\alpha>\rho(1)$.

c) If $\rho(0)>0$, then $\beta(\alpha)=0$ for some $\alpha>0$, so by b) we have that $X=[\rho(u)]^\alpha\subset [\rho(u)]^0\subset X$.\\
Now assume that $\rho(0)=0$. This means that $\beta(\alpha)>0$ for all $\alpha\in(0,\rho(1)]$ and hence by b) we have that
$$
[\rho(u)]^0=\overline{\bigcup_{n=1}^\infty[\rho(u)]^{\alpha_n}}=\overline{\bigcup_{n=1}^\infty[u]^{\beta(\alpha_n)}}.
$$
where $(\alpha_n)\subset (0,\rho(1)]$ is such that $\alpha_n\searrow 0$.
If $\rho(r_+)>0$, then $\beta(\alpha_n)=r_+$ for sufficiently large $n$, so $[\rho(u)]^0=[u]^{r_+}$ in this case.\\
If $\rho(r_+)=0$, then
\begin{equation}\label{f1}\forall \, n\in\N, \, \;\beta(\alpha_n)>r_+\mbox{ and }\beta(\alpha_n)\to r_+.
\end{equation} By definition, $\beta(\alpha_n)\geq r_+$ for all $n\in\N$. If $\beta(\alpha_n)=r_+$ for some $n$, then by a) we have a contradiction $\rho(r_+)=\rho(\beta(\alpha_n))\geq \alpha_n>0$. On the other hand, assume that $\beta(\alpha_n)$ does not converge to $r_+$. Then for some $r'>r_+$, we have that $\beta(\alpha_n)>r'$ for all $n\in\N$ (because $(\beta(\alpha_n))$ is nonincreasing), which implies that $\{t:\rho(t)>0\}\subset [r',\infty)$. This contradicts the definition of $r_+$. Hence we get (\ref{f1}). This condition together with the fact that $s>t\Rightarrow [u]^t\subset[u]^s$ gives the thesis of (c) in this case.
\end{proof}

We notice that, in a metric space $(X,d)$ we have $u=\bchi_{K}  \in \mathcal{F}_{X}^*$ if and only if $K$ is a compact and nonempty subset of $X$. In that case, $[\rho(u)]^{0}=K$. Also $[1]^{0}=X$.

\begin{proposition}\label{equivalence ndrc}
If $\rho:[0,1]\to[0,1]$ is ndrc, then for every $u\in\mathcal{F}^*_X$, the fuzzy set $\rho(u)$ is usc.
\end{proposition}
\begin{proof}
The assertion follows directly from Proposition \ref{alpha cut compose with ndrc rho} b) and c).
\end{proof}

\begin{definition}
\emph{An iterated function system (IFS) is a finite family $\mathcal{R}$ of continuous functions $\phi_j:X\to X$, denoted by $\mathcal{R}=(X,(\phi_j)_{j=0,...,n-1})$. If additionally the mappings $\phi_j$ satisfy
$$
d(\phi_j(a),\phi_j(b))\leq \; \lambda_j \; d(a,b), \; j=0...n-1,
$$
for some constants $0\leq \lambda_j<1$, $j=0,...,n-1$, then we call it Lipschitz contractive IFS.\\
The operator $\mathcal{R}:\mathcal{K}^*(X)\to\mathcal{K}^*(X)$ defined by $$\displaystyle \mathcal{R}(B):= \bigcup_{j=0...n-1} \phi_j(B)$$
is called }the Hutchinson-Barnsley (HB) operator associated to $\mathcal{R}$.
\end{definition}

\begin{definition}\label{admissible gey level maps}  \emph{A
 system of grey level maps $(\rho_j)_{j=0...n-1}: [0,1]\to [0,1]$ is }admissible\emph{ if it satisfies all the conditions\\
a) $\rho_j$ is nondecreasing;\\
b) $\rho_j$ is right continuous;\\
c) $\rho_j(0)=0$;\\
d) $\rho_j(1)=1$ for some $j$.}
\end{definition}

The items a) and  b) mean that each grey level map is ndrc. Item  c) means that black pixels should stay black and item d) means that the combination of the grey scales cannot decrease the brightness, { when we represent fuzzy sets as grey scale images.}\\
\begin{figure}[h!]
  \centering
  \includegraphics[width=3.5cm]{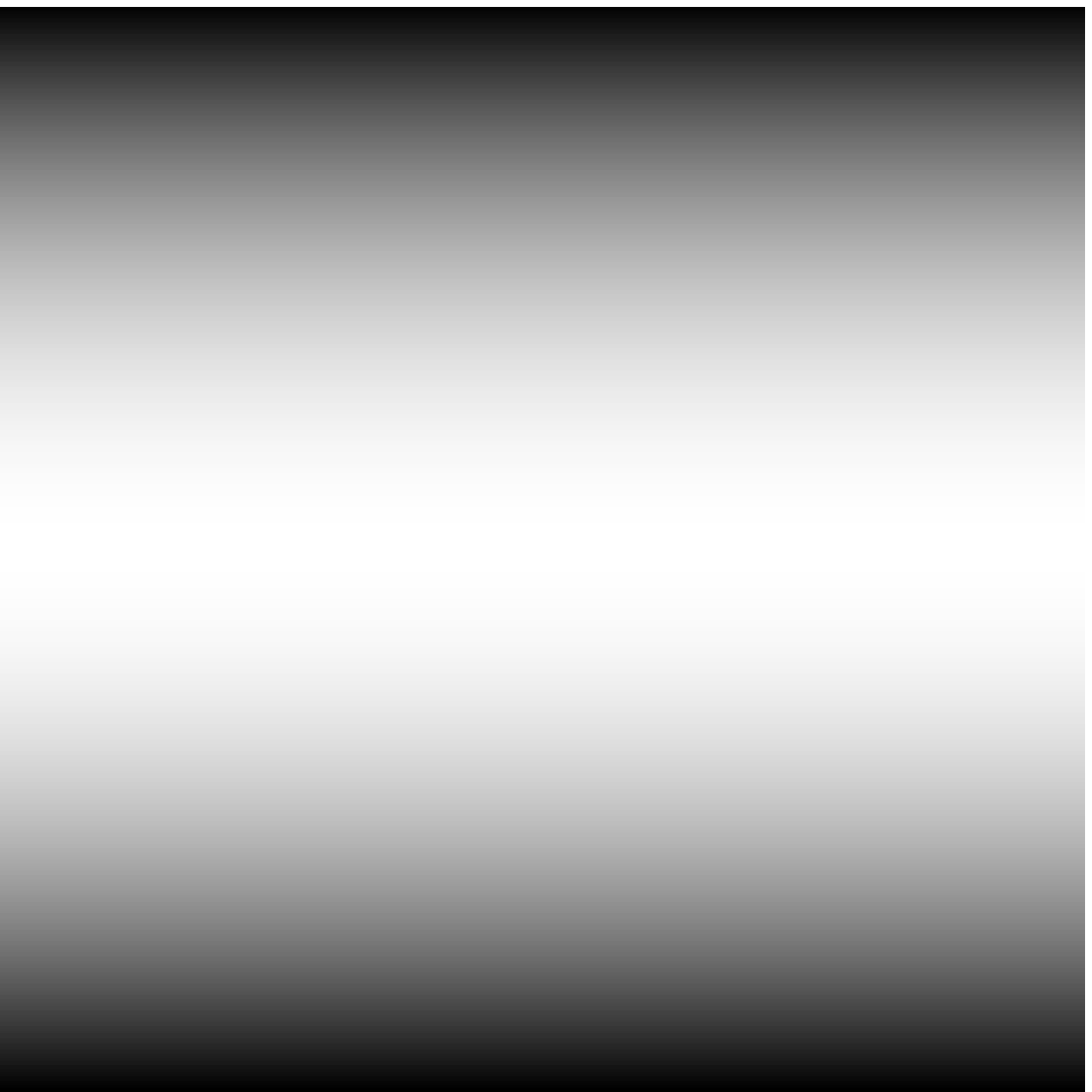}
  \includegraphics[width=5.0cm]{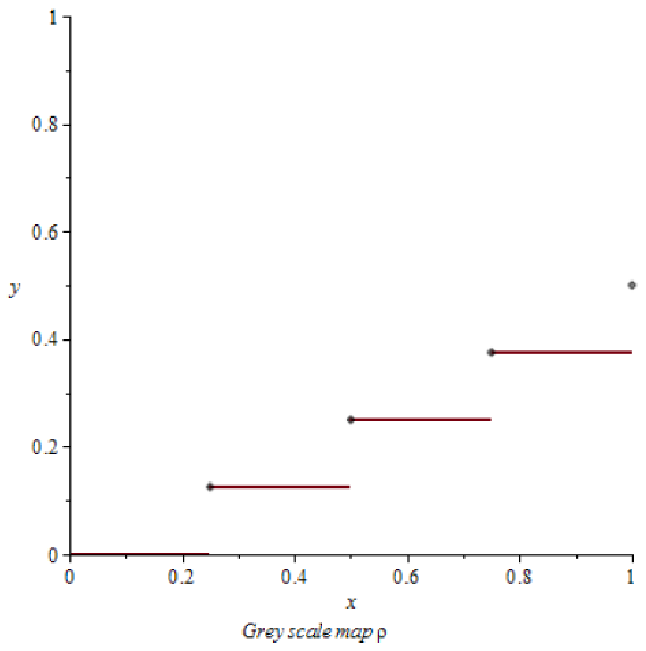}
  \includegraphics[width=3.5cm]{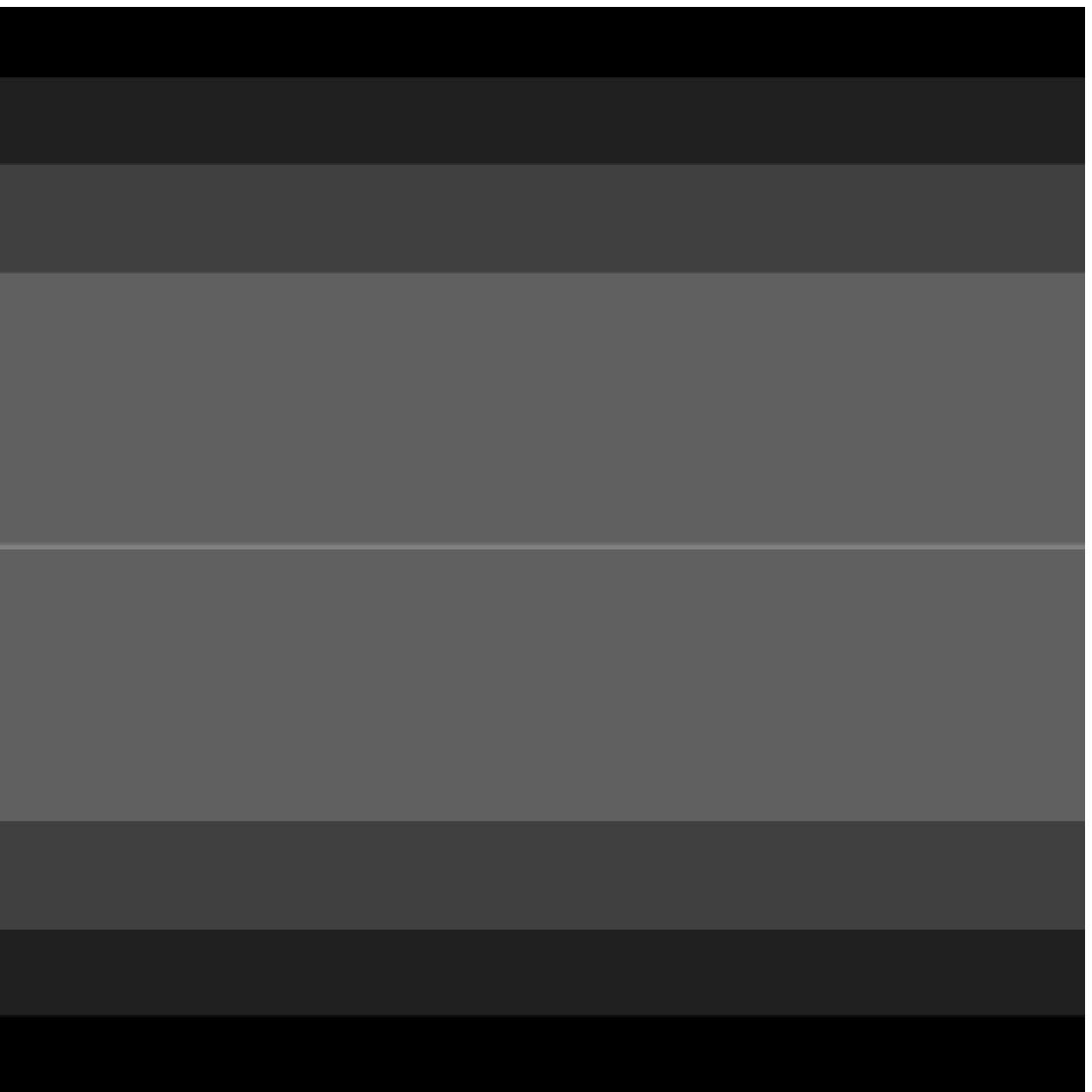}
  \caption{On the left we have the fuzzy set $u(x,y)=1- 4(y- \frac{1}{2})^2$ in $X=[0,1]^2$ as a grey scale figure and on the right $\rho(u)$, where $\rho(t)=1/8 (4t - frac(4t))$ is drawn in the middle.}
  \label{example fuzzy set compose}
\end{figure}

The fuzzification of an IFS is to consider the parallel action of the Hutchinson-Barnsley operator on the fuzzy subsets of $X$.

\begin{definition}\label{ifzs definition} \emph{ Let $\mathcal{R}=(X, (\phi_j)_{j=0...n-1})$ be an IFS and $(\rho_j)_{j=0...n-1}$ be an admissible
system of grey level maps. Then the system $\mathcal{Z_R}:=(X, (\phi_j)_{j=0...n-1}, (\rho_j)_{j=0...n-1})$ is called} an iterated fuzzy function system\emph{ (IFZS in short). Inspired by the (HB) operator, we define }the Fuzzy Hutchinson-Barnsley (FHB) operator associated to $\mathcal{Z_R}$\emph{ by
$$\mathcal{Z_R}(u):= \bigvee_{j=0...n-1} \rho_j(\phi_j(u))$$
for all $u\in\mathcal{F}^*_X$.}
\end{definition}

\begin{example} To see the action of the FHB operator we consider the IFS $\mathcal{R}=([0,1]^2, (\phi_j)_{j=0,1})$, where $\phi_0(x,y)=(x/2, y/2)$ and $\phi_1(x,y)=(x/2, y/2 + 1/2)$, with an admissible
set of grey level maps $\rho_0 (t)= 1/3 (3t - frac (3t))$, $\rho_1 (t)= 0$ if $t<1/2$ and $\rho_1 (t)= t$ if $t \geq 1/2$. In the Figure~\ref{example fhb operator graph} we drawn $\mathcal{Z_R}(u)$ for a representation of the fuzzy set $u(x,y)=\frac{x+y}{2}$ in $X=[0,1]^2$ as a grey scale figure.
\begin{figure}[h!]
  \centering
  \includegraphics[width=13cm]{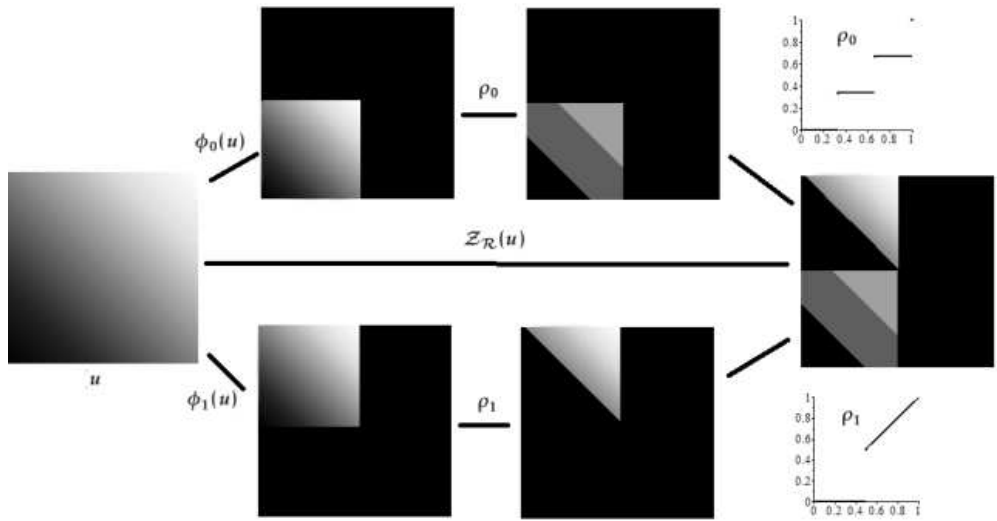}
  \caption{ The FHB operator acting on $u$.}\label{example fhb operator graph}
\end{figure}
\end{example}

\begin{proposition} \label{grey level map preserves *}
 If $u \in \mathcal{F}_{X}^*$ then  $\mathcal{Z_R}(u) \in \mathcal{F}_{X}^*$.
\end{proposition}
\begin{proof}
From Proposition \ref{compose usc plus cont}, $\phi_j(u) \in \mathcal{F}_{X}^*$ for any $j=0...n-1$, because each $\phi_j$ is Lipschitz continuous. Moreover, $\rho_j(\phi_j(u))$ are usc because  the grey level maps are admissible (see Proposition~\ref{equivalence ndrc}). Finally, $\rho_j(\phi_j(u))$ is compactly supported for each $j$ by Proposition~\ref{alpha cut compose with ndrc rho}(c). Thus $\mathcal{Z_R}(u)$ is usc and compactly supported as the supremum of usc and compactly supported maps.

From the admissibility of $\rho_j$ there is some $j_0$ such that $\rho_{j_{0}}(1)=1$. Since $\phi_{j_{0}}(u)$ is normal,  we can find $y_0$ such that $\phi_{j_0}(u)(y_0)=1$. Then
$$\mathcal{Z_R}(u)(y_0) \geq \rho_{j_{0}}(\phi_{j_{0}}(u)(y_0))= 1.$$
Thus, $\mathcal{Z_R}(u)$ is normal.
\end{proof}

The next lemma will be proved with more generality that we need here. Such a version will be useful in other applications.
\begin{lemma} \label{general properties of fuzzy composition} Let $X$ and $Y$ be  metric  spaces and $(\rho_j)_{j=0...n-1}$ an admissible family of grey level maps. Consider $\phi_{j}: X \to Y$ a family of continuous maps for $j=0...n-1$. Then, for each $u \in \mathcal{F}_{X}^*$\\
\vspace{.3cm}
a) $q_{j}=\rho_{j}(\phi_{j}(u)) \in \mathcal{F}_{Y}$ is usc  and compactly supported;\\
\vspace{.2cm}
b) $[q_{j}]^{\alpha}=\phi_{j}([\rho_{j}(u)]^{\alpha})\;$ for every $\alpha\in[0,1]$;\\
\vspace{.2cm}
c) $\displaystyle \left[\bigvee_{j=0...n-1} \rho_j(\phi_j(u))\right]^{\alpha} = \bigcup_{j=0...n-1}\phi_{j}([\rho_{j}(u)]^{\alpha})\;$ for every $\alpha\in[0,1]$;\\
\vspace{.2cm}
d) $\displaystyle \bigvee_{j=0...n-1} \rho_j(\phi_j(u))$ is normal.
\end{lemma}
\begin{proof} The proof follows exactly the same reasoning as in \cite{MR1192494}, except by a) that is a consequence of Proposition~\ref{compose usc plus cont}  (and can be proved similarly as Proposition \ref{grey level map preserves *}), and d) which also can be proved similarly as in Proposition \ref{grey level map preserves *}.
\end{proof}

The following result is an extension of classical Hutchinson-Barnsley Theorem for IFZS. We skip the proof as it is the same as in the particular case of compact $X$ proved in \cite{MR1192494} (also, later we will prove much more general result).
\begin{theorem}\label{cabrelli fixed point fractal operator}
 Given a contractive IFZS $\mathcal{Z_R}=(X,(\phi_j)_{j=0,...,n-1},(\rho_j)_{j=0,...,n-1})$,  the FHB operator $\mathcal{Z_R}: \mathcal{F}_{X}^* \to \mathcal{F}_{X}^*$ is a Banach contraction in $(\mathcal{F}_{X}^*, d_\infty)$. More precisely,
$$d_\infty(\mathcal{Z_R}(u),\mathcal{Z_R}(v)) \leq \lambda \; d_\infty(u,v), \; \forall u,v \in \mathcal{F}_{X}^*,$$
where $\lambda:=\max\{Lip(\phi_j):j=0,...,n-1\}$ and $Lip(\phi_j)$, $j=0,...,n-1$ are contraction constants of $\phi_0,...,\phi_{n-1}$, respectively.\\
In particular, if $X$ is complete, then there exists a unique $u^* \in \mathcal{F}_{X}^*$ such that $$\mathcal{Z_R}(u^*)=u^*$$ and, moreover, for any $v \in \mathcal{F}_{X}^*$ we get
$$d_\infty(\mathcal{Z_R}^{(k)}(v), \; u^*) \to 0$$
where $\mathcal{Z_R}^{(k)}(v)$ denotes the $k$-th iteration of the (FHB) operator $\mathcal{Z_R}$.
\end{theorem}
%

\begin{definition}\emph{The fuzzy set $u^*$ from the above theorem is called} the fuzzy attractor \emph{or} fuzzy fractal generated by IFZS $\mathcal{Z_R}$.
\end{definition}
\begin{remark}\emph{Directly from the definition and Lemma~\ref{general properties of fuzzy composition} it follows that if $u^*$ is the fuzzy fractal generated by a IFZS $\mathcal{Z_R}=(X,(\phi_j)_{j=0,...,n-1},(\rho_j)_{j=0,...,n-1})$, then for every $\alpha\in[0,1]$,
$$[u^*]^{\alpha} = \bigcup_{j=0...n-1}\phi_{j}([\rho_{j}(u^*)]^{\alpha})$$
}\end{remark}
The next result is known as the IFZS collage theorem.
\begin{theorem} Assume that $X$ is complete and let $\mathcal{R}=(X, (\phi_j)_{j=0...n-1})$ be a contractive IFS with contraction constant $\lambda=\max_{j} \mathrm{Lip}(\phi_j) < 1$ and $u^* \in \mathcal{F}_{X}^*$ be the fuzzy attractor of the IFZS $\mathcal{Z_R}=(X, (\phi_j)_{j=0...n-1}, (\rho_j)_{j=0...n-1})$. If $v \in \mathcal{F}_{X}^*$  then
$$d_{\infty}(v, u^*) < \frac{1}{1-\lambda} \; d_{\infty}(v, \mathcal{Z_R}(v)).$$
\end{theorem}
\begin{proof} The proof is a combination of the following results. From Theorem~ \ref{Fuzzy Space is Complete} we get that $(\mathcal{F}_{X}^*, d_{\infty})$ is complete. From Theorem~\ref{cabrelli fixed point fractal operator} we get that $\mathcal{Z_R}$ is a Lipschitz contraction with $\lambda=\mathrm{Lip}(\mathcal{Z_R})$ and we also have the existence of the fuzzy fractal attractor $u^*$. So our result follows from the standard Collage Theorem~\ref{Collage Theorem}.
\end{proof}

We end this section with presenting some further properties of IFZS.
\begin{definition}\emph{
Given $u, v  \in \mathcal{F}_{X}^*$, we say that $ u \leq v$ if $u(x) \leq v(x)$ for all $x \in X$}.
\end{definition}
\begin{lemma}
If $\mathcal{Z_R}$ is an IFZS, then the associated operator $\mathcal{Z_R}$  is monotone that is $\mathcal{Z_R}(u) \leq \mathcal{Z_R}(v)$ if $u \leq v$.
\end{lemma}
\begin{proof} If $u \leq v$ then for every $j=0,...,n-1$ and $z\in X$, we have
$$ \phi_{j}(u)(z)=\sup_{\phi_{j}(y)=z} u(y) \leq \sup_{\phi_{j}(y)=z} v(y)=\phi_{j}(v)(z) $$
provided $\phi_j^{-1}(z)\neq\emptyset$, and
$$
\phi_j(u)(z)=0=\phi_j(v)(z)
$$
in the opposite case. Hence $\phi_j(u)\leq \phi_j(v)$, and thus we also have
$$\mathcal{Z_R}(u)= \bigvee_{j=0...n-1} \rho_j(\phi_j(u)) \leq \bigvee_{j=0...n-1} \rho_j(\phi_j(v)) = \mathcal{Z_R}(v),$$
because the grey scale maps are nondecreasing.
\end{proof}
\begin{proposition} \label{Monotone Attractor and Fuzzy Attract} Let the crisp set $A(\mathcal{R}) \in \mathcal{K}^*(X)$ be the attractor of an IFS $\mathcal{R}=(X, (\phi_j)_{j=0...n-1})$, and $u^* \in \mathcal{F}_{X}^*$ be the fuzzy attractor of the IFZS $\mathcal{Z_R}=(X, (\phi_j)_{j=0...n-1}, (\rho_j)_{j=0...n-1})$. Then, for any  $B \in \mathcal{K}^*(X)$ and $v \in \mathcal{F}_{X}^*$ we have\\
a) if $\mathcal{Z_R}(v) \leq v$ then $u^* \leq v$;\\
b) if $\mathcal{R}(B) \subseteq B$ then $A(\mathcal{R}) \subseteq B$;\\
c) if $v \leq \mathcal{Z_R}(v)$ then $v \leq u^*$;\\
d) if $B \subseteq \mathcal{R}(B) $ then $B \subseteq A(\mathcal{R})$.\\
\end{proposition}
\begin{proof}
a) If $\mathcal{Z_R}(v) \leq v$ we get $\mathcal{Z_R}^{(k)}(v) \leq v$ for $k \geq 1$, taking the limit and using Theorem~\ref{cabrelli fixed point fractal operator} we get  $\mathcal{Z_R}^{(k)}(v) \to u^*$. Hence $u^* \leq v$. Indeed, similarly as in \cite{MR1192494} we can show that $\left\{u\in\mathcal{F}^*_X:u\leq v\right\}$ is closed in $\mathcal{F}^*_X$. The other items are proved in the same fashion.
\end{proof}

\begin{theorem} \label{zero cut IFZS versus IFS} Let the crisp set $A(\mathcal{R}) \in \mathcal{K}^*$ be the attractor of the IFS $\mathcal{R}=(X, (\phi_j)_{j=0...n-1})$, and $u^* \in \mathcal{F}_{X}^*$ be the fuzzy attractor of the IFZS $\mathcal{Z_R}=(X, (\phi_j)_{j=0...n-1}, (\rho_j)_{j=0...n-1})$. Then $[u^*]^{0} \subseteq A(\mathcal{R})$.
\end{theorem}
\begin{proof} Consider $\bchi_{A(\mathcal{R})} \in \mathcal{F}_{X}^*$. From Remark~\ref{Crisp Set Extension to Fuzzy Set}, for any $z\in X$, we have
$$ \phi_{j}(\bchi_{A(\mathcal{R})})(z)=
\bchi_{\phi_{j}(A(\mathcal{R}))}(z) \leq \bchi_{A(\mathcal{R})}(z),$$
because $A(\mathcal{R}) = \cup \phi_{j}(A(\mathcal{R}))$ implies $\phi_{j}(A(\mathcal{R})) \subseteq A(\mathcal{R})(z)$ for any $j$. In particular $\rho_j(\phi_{j}(\bchi_{A(\mathcal{R})})(z)) \leq  \rho_j( \bchi_{A(\mathcal{R})} (z)) \leq \bchi_{A(\mathcal{R})}(z)$ for any $j$, because  $\rho_j(0)=0$ and $\rho_j(1) \leq 1$.  Since $\mathcal{Z_R}(u)= \bigvee_{j=0...n-1} \rho_j(\phi_j(u))$ is a supremum we get $\mathcal{Z_R}(\bchi_{A(\mathcal{R})}) \leq \bchi_{A(\mathcal{R})}$. From Proposition~\ref{Monotone Attractor and Fuzzy Attract} we get $u^* \leq \bchi_{A(\mathcal{R})}$ thus $[u^*]^{0} \subseteq A(\mathcal{R})$, because if $x \not\in A(\mathcal{R})$ then $ 0= \bchi_{A(\mathcal{R})}(x) \geq u^*(x) \geq 0$ so $u^*(x) =0$.
\end{proof}

\subsection{Generalized iterated function systems}



In this section we recall the theory of generalized iterated function systems introduced by Miculescu and Mihail in 2008.

Let $(X,d)$ be a metric space and $m\in\N$. By $X^m$ we denote the Cartesian product of $m$ copies of $X$, considered as a metric space with the maximum metric $d^m$:
\begin{equation}\label{d^m}
d^m((x_0,...,x_{m-1}),(y_0,...,y_{m-1})):=\max\{d(x_0,y_0),...,d(x_{m-1},y_{m-1})\},\;\;\;(x_0,...,x_{m-1}),(y_0,...,y_{m-1})\in X^m.
\end{equation}
It turns out that appropriately contractive GIFSs generates fractals sets. In order to formulate the result we need some further notation.
\begin{definition}\emph{
We say that $f:X^m\to X$ is }a generalized Matkowski contraction of degree $m$\emph{, if for some nondecreasing $\varphi:[0,\infty)\to [0,\infty)$ with $\varphi^{(k)}(t)\to 0$ for $t>0$ (here $(\varphi^{(k)}(t))$ is the sequence of iterations of $\varphi$ at the point $t$), it holds
\begin{equation*}
d(f(x),f(y))\leq\varphi(d^m(x,y)),\;\;x,y\in X^m
\end{equation*}
A function $\varphi$ is called }a witness for $f$.
\end{definition}
\begin{remark}\emph{
(1) It is easy to see that if $Lip(f)<1$, then $f$ is a generalized Matkowski contraction - the function $\varphi(t):=Lip(f)\cdot t$ is a witness.\\
(2) If $m=1$, then a generalized Matkowski contraction is called a Matkowski contraction, and it is known that each Matkowski contraction on a complete metric space satisfies the thesis Banach Fixed Point theorem (see Matkowski~\cite{MR0412650}). In fact, it is one of the strongest generalizations of the Banach Fixed Point theorem. For comparison of other notions of contractiveness, we refer the reader to a paper \cite{MR2338580}.}
\end{remark}
The next result shows that the (mentioned above) Matkowski fixed point theorem can be extended to generalized Matkowski contractions. For a proof, see Strobin and Swaczyna \cite{MR3263451}, \cite{MR3011940} and Mihail and Miculescu~\cite{MR2415407}, Theorem 3.4 (for a weaker case).
\begin{theorem}
 \label{Generalized Fixed Point Theorem} Let $(A,d)$ be a complete metric space. Given a generalized Matkowski contraction $F: A^m \to A$, there exists a unique $a \in A$ such that
$F(a, ..., a)= a.$
Moreover, for every $a_0, a_1, . . . , a_{m-1} \in A$, the sequence $a_k, \; k \geq 0$ defined by
$$a_{k+m}=F(a_{k+ m-1}, a_{k+ m-2}, . . . , a_{k}),$$
for all $k \in \mathbb{N}$, is convergent to $a$.
\end{theorem}
Also in Strobin \cite{MR3263451} (see also Mihail~\cite{MR2568892}), we have
\begin{proposition} \label{Mihail Dist  Unions}
Let  $(X, d)$ a  metric space and $f_j:X^m\to X$, $0,...,n-1$ be generalized Matkowski contractions, with witnessing functions $\varphi_j, \;j=0,...,n-1$. Then the map $F: \mathcal{K}^{*}(X)^m \to \mathcal{K}^{*}(X)$ given by $$F(H_{0}, ..., H_{m-1})=\bigcup_{j=0}^{n-1} f_j(H_{0}\times \cdots \times H_{m-1}),$$
is a generalized Matkowski contraction with witness function $\varphi=\max_{j}\varphi_j$.\\
In particular, if $f_j$ is Lipschitz contractive with $Lip(f_j)<1, \;j=0,...,n-1$, then $F$ is also Lipschitz contractive and $Lip(F)\leq \max_{j} Lip(f_j)<1$.
\end{proposition}
Now let us recall some properties of the Hausdorff distance. A proof can be found for example in ~\cite{MR2415407}, Proposition 2.7.
\begin{proposition}\label{filip1}
 If $(Y, d')$ and $(Z, d'')$ are metric spaces then  \\
i) if $H$ and $K$ are non empty sets of $Y$ then $h(H,K)=h(\overline{H},\overline{K})$;\\
ii) if $H_i$ and $K_i$ for $i \in I$, are non empty families of  sets of $Y$ then $\displaystyle h\left(\bigcup_{i \in I} H,\bigcup_{i\in I} K\right) \leq \sup_{i \in I} h(H_{i},K_{i})$.
\end{proposition}

We are ready to define generalized iterated function systems and prove the existence theorem (see mentioned papers \cite{MR3180942}, \cite{MR2568892}, \cite{MR2415407},\cite{MR3263451} and \cite{MR3011940}).

\begin{definition}
A generalized iterated function system of degree $m$ \emph{ (GIFS) is a (finite) family $\mathcal{S}$ of  continuous mappings $\phi_j: X^{m} \to X$, denoted $\mathcal{S}=(X, (\phi_j)_{j=0...n-1}).$\\
If each $\phi_j$ is a generalized Matkowski contraction, then we say that $\mathcal{S}$ is }Matkowski contractive.\emph{\\
If each $Lip(\phi_j)<1$, then we say that $\mathcal{S}$ is }Lipschitz contractive\emph{.\\
The operator $\mathcal{S}:\mathcal{K}^*(X)^m\to\mathcal{K}^*(X)$ defined by
$$
\mathcal{S}(K_0,...,K_{m-1}):=\bigcup_{j=0,...,n-1}\phi_j(K_0\times...\times K_{m-1})
$$
is called }the generalized Hutchinson-Barnsley (GHB) operator associated to $\mathcal{S}$.
\end{definition}

\begin{theorem} \label{GHB fuzzy attract exist and uniq} Given a Matkowski contractive GIFS $\mathcal{S}(X,(\phi_i)_{i=0,...,n-1})$ of order $m$, the GHB operator is a generalized Matkowski contraction. In particular, if $X$ is complete, then there exists a unique $A_\mathcal{S}\in\mathcal{K}^*(X)$ such that
$$
A_\mathcal{S}=\mathcal{S}(A_\mathcal{S},...,A_\mathcal{S})=\bigcup_{j=0,...,n-1}\phi_j(A_\mathcal{S}\times...\times A_\mathcal{S})
$$
Moreover, for every $K_0,...,K_{m-1}\in\mathcal{K}^*(X)$, the sequence $(K_k)$ defined by
$$
K_{k+m}=\mathcal{S}(K_{k},...,K_{k+m-1})
$$
converges to $A_\mathcal{S}$.
\end{theorem}
\begin{proof}
By the above auxiliary results we see that $\mathcal{S}:\mathcal{K}^*(X)^m\to\mathcal{K}^*(X)$ is a generalized Matkowski contraction (or even $Lip(\mathcal{S})<1$, if $\mathcal{S}$ is Lipschitz contractive), so the result follows from Theorem \ref{Generalized Fixed Point Theorem}.
\end{proof}



\section{GIFS fuzzyfication}
In this section we introduce and study a fuzzy version of GIFSs.\\
Let $X$ be a metric space and $m\in\N$. We start with recalling the definition of the finite Cartesian product of fuzzy sets.

\begin{definition} Given $m \geq 2$ and $u_{0}, ..., u_{m-1} \in \mathcal{F}_{X}$ we define we define the Cartesian product  $u_{0}\times ...\times u_{m-1} \in \mathcal{F}_{X^m}$ by
$$(u_{0}\times ...\times u_{m-1})(x_0, ..., x_{m-1})= \bigwedge_{i=0}^{m-1 }u_{i}(x_{i}).$$
\end{definition}

\begin{figure}[h!]
  \centering
  \includegraphics[width=2cm]{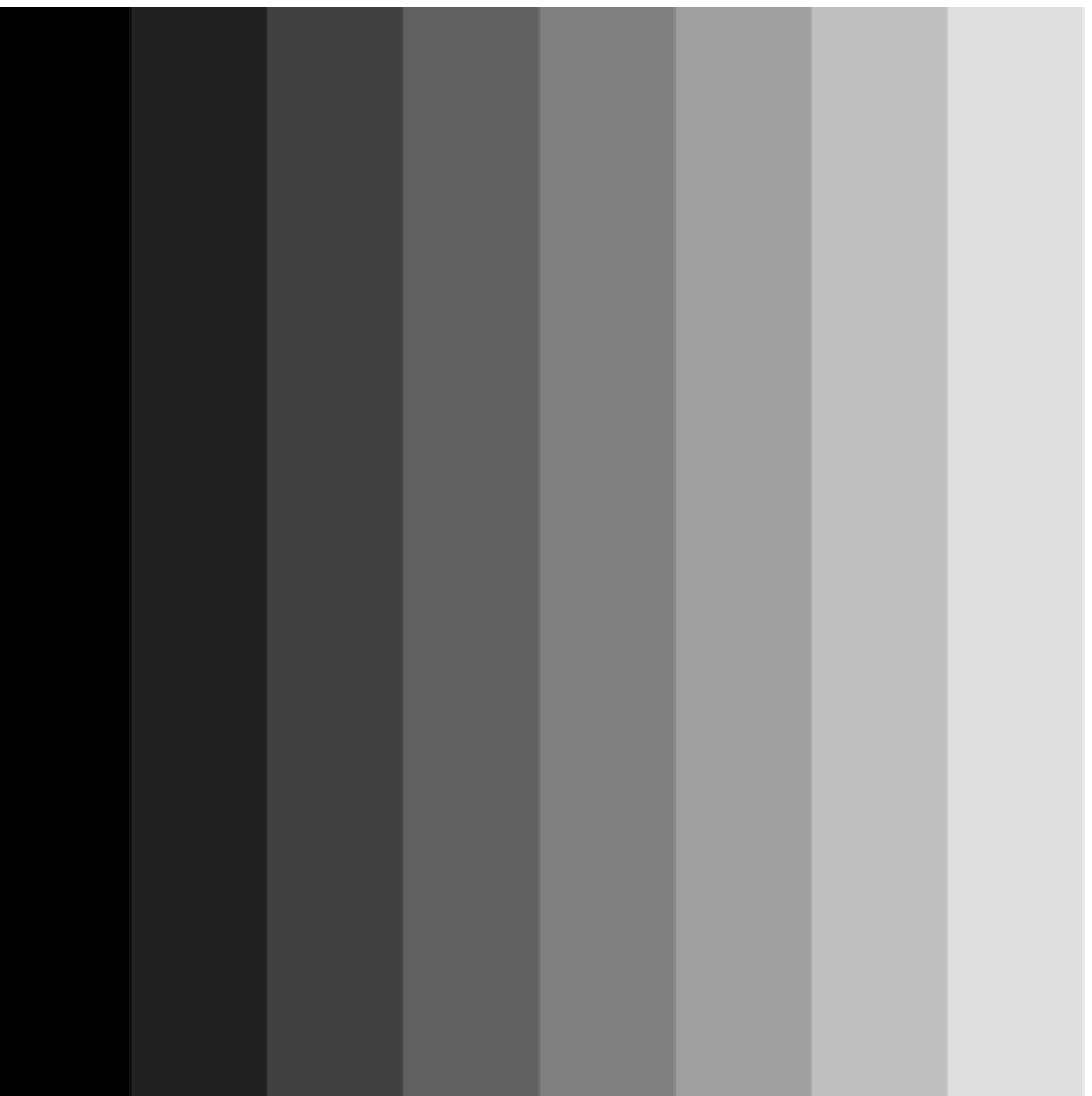}
  \includegraphics[width=2cm]{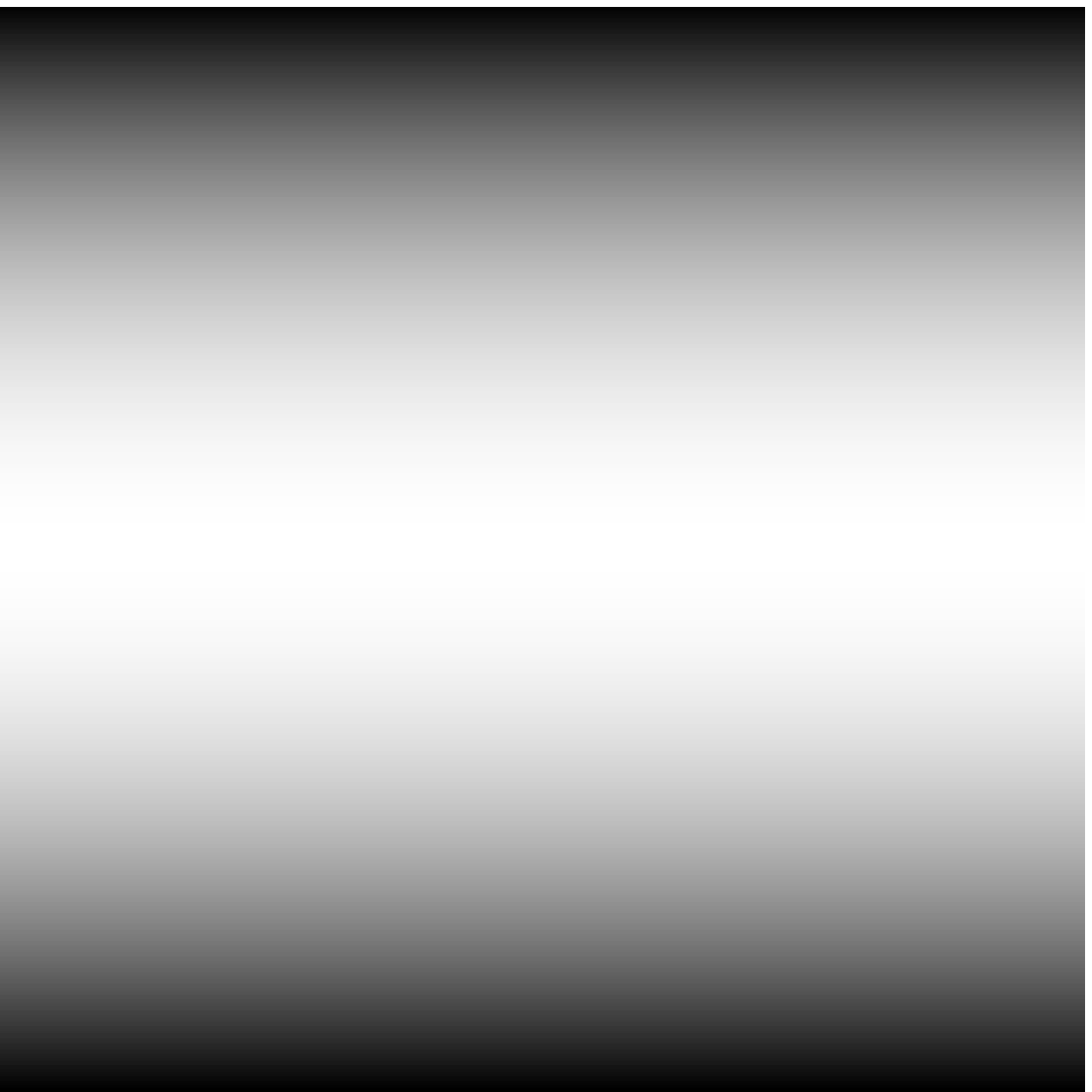}\\
  \includegraphics[width=3cm]{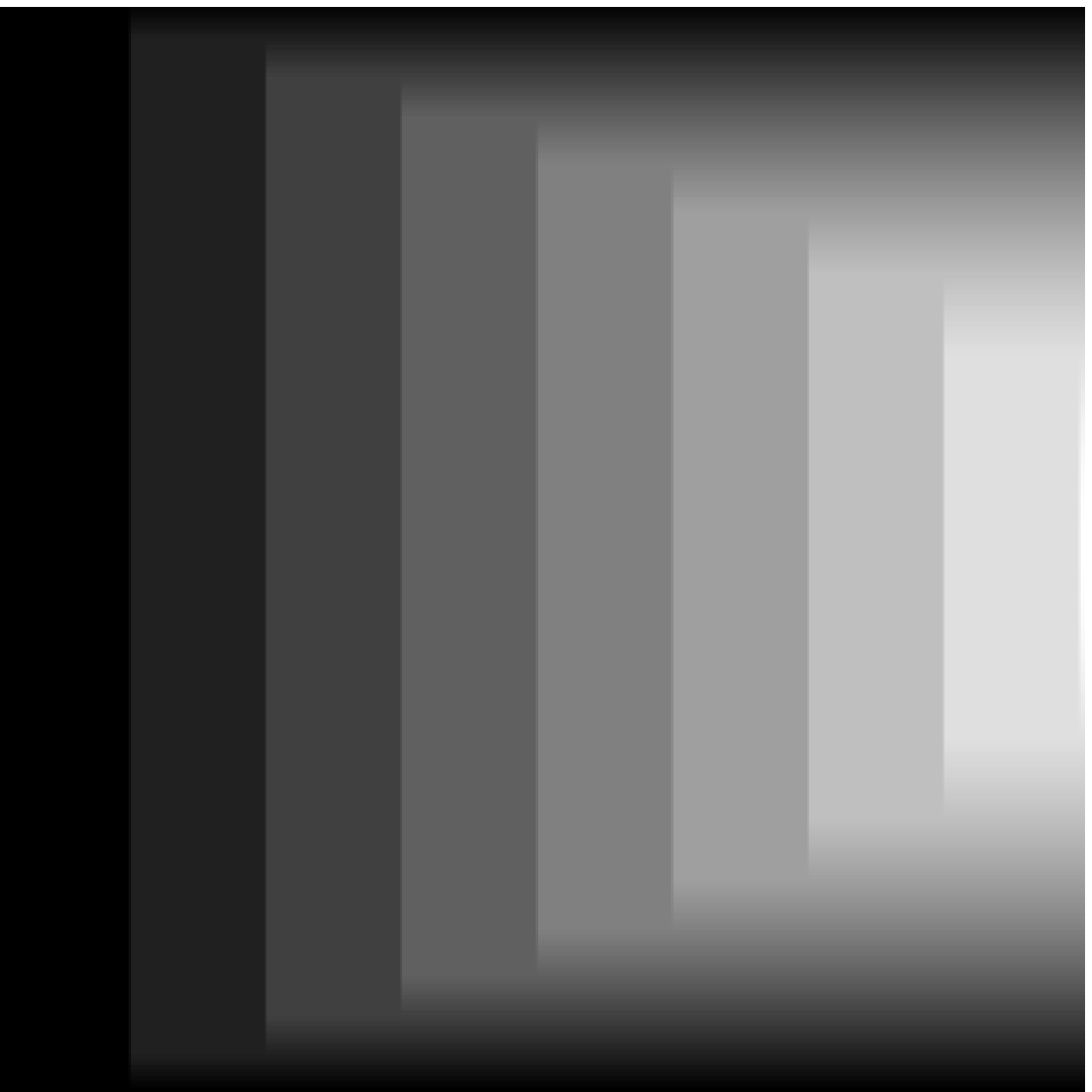}
  \caption{Representation of the fuzzy set $u \times v$ in $X=[0,1]^2$ as a grey scale figure (in the bottom). In the top, $u(x)=1/8 (8x - frac(8x))$ (left) and $v(y)=1- 4(x- \frac{1}{2})^2$ (right) are actually fuzzy sets in $X=[0,1]$ but we represent in $X=[0,1]^2$ to get a better graphical idea ($u(x)=u(x,y)$ and $v(y)=v(x,y)$).}\label{example fuzzy set cartesian}
\end{figure}

\begin{remark}
\textbf{We denote  $$(u_{0}, ..., u_{m-1})= (u)_{i=0}^{m-1}=(u)_i \text{  and }u_{0}\times ... \times u_{m-1} = \btimes_{i=0}^{m-1}  u_{i} = \btimes_{i}  u_{i}$$ to simplify the writing of elements in $(\mathcal{F}_{X}^*)^m$, if there is no risk of misunderstanding.}
\end{remark}

\begin{proposition}\label{cartesian product is usc and normal} Consider $u_{0}, ..., u_{m-1}  \in \mathcal{F}_{X}$.\\
a) If $u_{i}, \; i=0,..,m-1$ are normal, then $\btimes_{i}  u_{i}$ is normal.\\
b) If $u_{i}, \; i=0,..,m-1$ are usc, then $\btimes_{i}  u_{i}$ is usc.\\
c) If $u_i,\; i=1,...,m-1$ are compactly supported, then $\btimes_{i}  u_{i}$ is compactly supported.
\end{proposition}
\begin{proof}\\
a) If $u_{i}(\bar{x}_{i}) =1, \; i=0,..,m-1$ then $(\btimes_{i}  u_{i}) (\bar{x}_{0}, ..., \bar{x}_{m-1})= 1$.\\
b) We claim that $$(\btimes_{i}  u_{i})^{-1}([c, +\infty])= u_{0}^{-1}([c, +\infty]) \times \cdots  \times u_{m-1}^{-1}([c, +\infty])$$
which will imply that $\btimes_{i}  u_{i}$ is usc as the latter set is obviously closed.

To see that our claim is true, we take $(\bar{x}_{0}, ..., \bar{x}_{m-1}) \in (\btimes_{i}  u_{i})^{-1}([c, +\infty])$. Then $$(\btimes_{i}  u_{i}) (\bar{x}_{0}, ..., \bar{x}_{m-1}) \geq c$$ that is $\bigwedge_{i=0}^{m-1 }u_{i}(\bar{x}_{i}) \geq c$ so $u_{i}(\bar{x}_{i}) \geq c$ for each $i=0,...,m-1$. Thus $(\bar{x}_{0}, ..., \bar{x}_{m-1}) \in u_{0}^{-1}([c, +\infty]) \times \cdots  \times u_{m-1}^{-1}([c, +\infty])$. The reciprocal it is true because the  minimum is the maximum of the lower bounds.\\
c) Similarly as in b) we can show that $$(\btimes_{i}  u_{i})^{-1}((0, +\infty])= u_{0}^{-1}((0, +\infty]) \times \cdots  \times u_{m-1}^{-1}((0, +\infty])$$
which clearly implies c).
\end{proof}

\begin{definition} \label{GIFZS definition}
A generalized iterated fuzzy function system of degree $m$ \emph{ (GIFZS in short) \\$\mathcal{Z_S}:=(X, (\phi_j)_{j=0...n-1}, (\rho_{j})_{j=0...n-1})$ consists of a GIFS $\mathcal{S}=(X, (\phi_j)_{j=0...n-1}),$ with a set of admissible grey level maps (see Definition~\ref{admissible gey level maps}) $(\rho_j)_{j=0...n-1}: [0,1]\to [0,1]$.\\
We say that ${\mathcal{Z_S}}$ is }Matkowski contractive \emph{or} Lipschitz contractive \emph{, if the GIFS $\mathcal{S}$ is so.\\
The operator $\mathcal{Z_S}:  \mathcal{F}_{X}^{*} \times \cdots \times \mathcal{F}_{X}^{*} \to \mathcal{F}_{X}^{*}$ defined by
$$\mathcal{Z_S}((u)_{i}):= \bigvee_{j=0...n-1} \rho_{j}(\phi_{j}(\btimes_{i}  u_{i}))$$
is called }the generalized fuzzy Hutchinson-Barnsley operator (GFHB) associated to $\mathcal{Z_S}$
\end{definition}

Recall that for $\phi_{j}: X^m \to X$ and $u_{0}, ..., u_{m-1} \in \mathcal{F}_{X}^*$ then $$\phi_{j}(\btimes_{i}  u_{i})(z)= \left\{
                                  \begin{array}{ll}
                                   \sup_{\phi_{j}((x)_{i})=z} \bigwedge_{i=0}^{m-1 }u_{i}(x_{i}), &\text{ if } \phi_{j}^{-1}(z) \neq \varnothing\\
                                    0,  & \text{otherwise}
                                  \end{array}
                                \right.
  .$$
\begin{proposition}
The operator  $\mathcal{Z_S}$ is well defined, that is, $\mathcal{Z_S} (\mathcal{F}_{X}^{*} \times \cdots \times \mathcal{F}_{X}^{*}) \subseteq \mathcal{F}_{X}^{*}$.
\end{proposition}
\begin{proof}
From Proposition~\ref{cartesian product is usc and normal}, we get $u_{0}\times ... \times u_{m-1}\in \mathcal{F}_{X^m}^{*}$ because $u_{0}, ..., u_{m-1}  \in \mathcal{F}_{X}^{*}$, and from Proposition~\ref{compose usc plus cont}  $\phi_{j}(u_{0}\times ... \times u_{m-1})$ is normal,  compactly supported  and usc. Moreover, $\rho_{j}(\phi_{j}(u_{0}\times ... \times u_{m-1}))$ is usc and compactly supported  because each $\rho_{j}$ is ndrc (see Lemma~\ref{general properties of fuzzy composition}  taking $X^m$ and $X$). Thus $\mathcal{Z_S} (u_{0}, ..., u_{m-1}) \in \mathcal{F}_{X}^{*}$ because the family $\rho_{j}$ is admissible implies that $\mathcal{Z_S} (u_{0}, ..., u_{m-1})$ is normal  (see Lemma  \ref{general properties of fuzzy composition}(d)).
\end{proof}

Next we consider the complete metric space $((\mathcal{F}_{X}^*)^m, d_\infty^{m})$, where $d_\infty^m$ is defined as in (\ref{d^m}), that is
$$d_{\infty}^m ((u_{0}, ..., u_{m-1}), (v_{0}, ..., v_{m-1}))=\max_{i=0,...,m-1} d_{\infty}(u_{i}, v_{i}).$$

\begin{lemma}\label{alpha cut cartesian product}  Given $\alpha \in [0,1]$ we get
$$
  \begin{cases}
    [(\times_{i}  u_{i})]^{\alpha}=\times_{i} [ u_{i}]^{\alpha}, &\mathrm{     \quad if } \quad \alpha >0; \\
    [(\times_{i}  u_{i})]^{0}=\times_{i} [ u_{i}]^{0}=\lim_{\beta_n \to 0}\times_{i} [ u_{i}]^{\beta_n}, & \quad \mathrm{ if } \quad\alpha =0  \quad \mathrm{ and  } \quad \beta_n>0.
  \end{cases}
$$
\end{lemma}
\begin{proof} With $\alpha>0$ we can deal similarly as in Proposition~\ref{cartesian product is usc and normal}, b), and obtain $$[(\times_{i}  u_{i})]^{\alpha}= (\times_{i}  u_{i})^{-1}([\alpha, +\infty])= \times_{i}  (u_{i})^{-1}([\alpha, +\infty])=\times_{i} [u_{i}]^{\alpha}.$$

For the second part, we show first $\times_{i} [ u_{i}]^{0}= [(\times_{i}  u_{i})]^{0}$.

Take $(x)_i \in [\times_{i} u_{i}]^{0}$ then there is a sequence $((y^k)_i)$ such that $(y^k)_i \to (x)_i$ with $(y^k)_i \in [\times_{i} u_{i}]^{\alpha_k}= \times_{i} [u_{i}]^{\alpha_k}$ for some $\alpha_k\searrow 0$. By the properties of the product topology, each coordinate $y^k_i \to x_i \in [u_{i}]^{0}$, so $(x)_i \in \times_{i}[u_{i}]^{0}$. Reciprocally, if $(x)_i \in \times_{i}[ u_{i}]^{0}$ there for every $i=0,...,m-1$ there is a sequence $(y_i^k)$ such that $y^k_i \to x_i$ and $y^k_i \in [u_{i}]^{\alpha_k^i}$. Take $\gamma_k:=\min_{i} \alpha_k^i >0$ and consider the sequence $((y^k)_i)$. We claim that $(y^k)_i \in\times_{i} [u_{i}]^{\gamma_k}= [\times_{i} u_{i}]^{\gamma_k}$. Indeed, $y^k_i \in [u_{i}]^{\alpha_k^i} \subseteq [u_{i}]^{\gamma_k}$ for $i=0,...,m-1$. Since $\displaystyle (x)_i= \lim_{k\to\infty} (y^k)_i$ we get $(x)_i \in [\times_{i} u_{i}]^{0}$.

Obviously $\displaystyle \times_{i} [ u_{i}]^{0}=\lim_{k \to \infty}\times_{i} [ u_{i}]^{\beta_k}$ for all $\beta_k\searrow 0$ because of the considered topology on $\mathcal{K}^*(X^m)$.
\end{proof}

\begin{remark}\label{f3'}\emph{
By the above Lemma, Proposition \ref{alpha cut compose with ndrc rho} and Lemma \ref{general properties of fuzzy composition}, we have that if $\mathcal{Z}_\mathcal{S}$ is a GIFZS and $u_0,...,u_{m-1}\in\mathcal{F}_X^*$, then for every $\alpha\in[0,1]$,
$$
\left[\mathcal{Z}_\mathcal{S}(\times_i u_i)\right]^\alpha=\bigcup_{j=0,...,n-1}\phi_j([\rho_j(\times_iu_i)]^\alpha)
$$
and if $\alpha\in(0,\rho_j(1)]$, then (note that since $\rho_j(0)=0$, be have $\beta_j(\alpha)>0$)
$$
[\rho_j(\times_iu_i)]^\alpha=[\times_iu_i]^{\beta_j(\alpha)}=\times_i[u_i]^{\beta_j(\alpha)}
$$
and if $\alpha=0$, then\\
$$
[\rho_j(\times_iu_i)]^0=\times_i[u_i]^{r^j_+}
$$ provided $\rho_j(r^j_+)>0$, and
$$
[\rho_j(\times_iu_i)]^0=\overline{\bigcup_{\alpha>r^j_+}[\times_iu_i]^\alpha}=\overline{\bigcup_{\alpha>r^j_+}\times_i[u_i]^\alpha}=\times_i\overline{\bigcup_{\alpha>r^j_+}[u_i]^\alpha}=\overline{\bigcup_{\alpha>r^j_+}[u_0]^\alpha}\times...\times \overline{\bigcup_{\alpha>r^j_+}[u_{m-1}]^\alpha}
$$provided $\rho_j(r^j_+)=0$.}
\end{remark}

The next lemma shows the relationship between the Hausdorff distance on $\mathcal{K}^*(X^m)$ and the maximum distance on $\mathcal{K}^*(X)^m$.
\begin{lemma} \label{property hausdorff cartesian product} Let $A_{0},..., A_{m-1}, B_{0},..., B_{m-1} \in \mathcal{K}^{*}(X)$ then
$$ h(\btimes_{i} A_{i}, \btimes_{i} B_{i})  = \max_{i=0,...,m-1} h(A_{i}, B_{i}).$$
\end{lemma}
\begin{proof} We recall that $U_{\varepsilon}=\{ x \in X\; | \; d(x, U)\leq \varepsilon\}$, and
$h(U,V)= \inf\{\varepsilon >0\; | \;U \subseteq V_{\varepsilon},  V \subseteq U_{\varepsilon}\}.$
We claim that $(\btimes_{i} A_{i})_{\varepsilon}= \btimes_{i} (A_{i})_{\varepsilon}$. Indeed, given $(z)_i \in (\btimes_{i} A_{i})_{\varepsilon}$ we get
$$\max_{i}d(z_i , x_i)= d_{m}((z)_i , (x)_i) \leq \varepsilon, \; \mathrm{\;for\;some\;} (x)_i \in \btimes_{i} A_{i}$$
that is $d(z_i , x_i)\leq \varepsilon$ and $x_i\in A_i$ for all $i=0,...,m-1$.
Thus $(z)_i \in \btimes_{i} (A_{i})_{\varepsilon}$. So $(\btimes_{i} A_{i})_{\varepsilon} \subseteq \btimes_{i} (A_{i})_{\varepsilon}$.  The reciprocal is evident, because $d_{m}((z)_i , (x)_i)= \max_{i}d(z_i , x_i)$.

Take $\varepsilon >0$ such that $\btimes_{i} A_{i} \subseteq (\btimes_{i} B_{i})_{\varepsilon}$ and $ \btimes_{i} B_{i} \subseteq (\btimes_{i} A_{i})_{\varepsilon}$. Then $\btimes_{i} A_{i} \subseteq \btimes_{i} (B_{i})_{\varepsilon}$ and $ \btimes_{i} B_{i} \subseteq \btimes_{i} (A_{i})_{\varepsilon}$. Using the properties of the Cartesian product we get $ A_{i} \subseteq  (B_{i})_{\varepsilon}, \; \forall i$ and $  B_{i} \subseteq  (A_{i})_{\varepsilon}, \; \forall i$. Thus,  $h(A_{i}, B_{i}) \leq \varepsilon$ for $i=0,...,m-1$, so $\max_{i=0,...,m-1} h(A_{i}, B_{i})\leq \varepsilon$. Reverting the reasoning above we get $ h(\btimes_{i} A_{i}, \btimes_{i} B_{i})  = \max_{i=0,...,m-1} h(A_{i}, B_{i}).$
\end{proof}

\begin{lemma}\label{metric cartesian product} The mapping $\psi:((\mathcal{F}_{X}^*)^m, d_{\infty}^m) \to  (\mathcal{F}_{X^m}^{*}, d_{\infty})$ given by $$\psi(u_{0}, ..., u_{m-1}):= u_{0}\times ...\times u_{m-1}$$ is isometry, that is
$$d_{\infty}(\psi(u_{0}, ..., u_{m-1}),\psi(v_{0}, ..., v_{m-1})) = d_\infty^{m}((u_{0}, ..., u_{m-1}),(v_{0}, ..., v_{m-1})).$$
\end{lemma}
\begin{proof}
By definition
$$d_{\infty}(\psi((u)_i),\psi((v)_i)) =\sup_{\alpha \in [0,1]} h\left([\times_{i}  u_{i}]^{\alpha}, [\times_{i}  v_{i}]^{\alpha}\right)$$

From Lemma~\ref{alpha cut cartesian product} we have $[\times_{i}  u_{i}]^{\alpha}= \times_{i}[u_{i}]^{\alpha}$ and from Lemma~\ref{property hausdorff cartesian product} we get, for any $\alpha \in [0,1]$,
$$
  h\left([\times_{i}  u_{i}]^{\alpha}, [\times_{i}  v_{i}]^{\alpha}\right) = h\left(\times_{i} [ u_{i}]^{\alpha}, \times_{i} [ v_{i}]^{\alpha}\right)\\
   =  \max_{i=0,...,m-1} h\left( [ u_{i}]^{\alpha}, [ v_{i}]^{\alpha}\right).
$$
So
\begin{align*}
d_{\infty}(\psi((u)_i),\psi((v)_i)) & = \sup_{\alpha \in [0,1]} \max_{i=0,...,m-1} h\left( [ u_{i}]^{\alpha}, [ v_{i}]^{\alpha}\right)\\
& = \max_{i=0,...,m-1} \sup_{\alpha \in [0,1]} h\left( [ u_{i}]^{\alpha}, [ v_{i}]^{\alpha}\right)\\
&=  \max_{i=0,...,m-1} d_{\infty}(u_i,v_i)\\
&= \max_{i} d_{\infty}(u_i,v_i)\\
&= \quad d_\infty^{m}((u_{0}, ..., u_{m-1}),(v_{0}, ..., v_{m-1})).\\
\end{align*}

\end{proof}
\begin{lemma}\label{rho composite with fuzzy cut} Let $(A, d)$ be a metric space and $\rho$ an ndrc  grey level map  with $\rho(0)=0$.  Then the map induced by $\rho$ is nonexpansive, that is,
$$d_{\infty}(\rho(u), \rho(v)) \leq d_{\infty}(u,v),$$
for any $u,v \in \mathcal{F}_{A}^*$.
\end{lemma}
\begin{proof}  By Corollary \ref{f2}, we have $\displaystyle d_{\infty}(\rho(u), \rho(v))= \sup_{0<\alpha \leq 1} h([\rho(u)]^{\alpha}, [\rho(v)]^{\alpha})$.

Take $0<\alpha \leq 1$. From Proposition~\ref{alpha cut compose with ndrc rho} we know that $[\rho(u)]^{\alpha}=\emptyset$ if $\alpha>\rho(1)$ and if $\alpha\in(0,\rho(1)]$, then $[\rho(u)]^{\alpha}=[u]^{\beta(\alpha)}$, where $\beta:[0,\rho(1)] \to [0,1]$, given by
$$\beta(\alpha)=\inf \{ t \; | \; \rho(t) \geq \alpha\}$$
is well defined and nondecreasing. Then
$$h([\rho(u)]^{\alpha}, [\rho(v)]^{\alpha})= h([u]^{\beta(\alpha)}, [v]^{\beta(\alpha)}) \leq d_{\infty}(u,v),$$
hence $$d_{\infty}(\rho(u), \rho(v)) \leq d_{\infty}(u,v).$$
\end{proof}

\begin{definition}\emph{
A fuzzy set $u\in \mathcal{F}_{X}^{*}$  is called} a generalized fuzzy fractal of a GIFZS $\mathcal{Z_S}=(X, (\phi_j), (\rho_{j}))_{j=0,...,n-1}$\emph{ if $\mathcal{Z_S}(u,...,u)=u$, that is
$$u= \bigvee_{j=0...n-1} \rho_{j}(\phi_{j}(\btimes_{i=0}^{m-1}u)).$$
}
\end{definition}

\begin{example} Consider $\mathcal{Z_S}=(\mathbb{S}^1 =\mathbb{R}/\mathbb{Z}\simeq [0,1], (\phi_j)_{j=0,1}, (\rho_{j})_{j=0,1}),$  with the admissible grey level maps $\rho_{j}(t):=t, \; t \in [0,1]$ and the Lipschitz maps given by $\phi_j(x,y):=\frac{1}{2} x + \frac{j}{2}, \; j=0,1$. In this case,
$$\phi_{j}(u \times v)(z)= \left\{
                                  \begin{array}{ll}
                                   \sup_{y \in [0,1], \;\frac{1}{2} x + \frac{j}{2} = z} u(x) \wedge v(y), &\text{ if } \exists x \;s.t. \;\frac{1}{2} x + \frac{j}{2} = z\\
                                    0,  & \text{otherwise}
                                  \end{array}
                                \right.
$$

Since $[0, 1]=[v]^{u(2z-j)} \cup [0, 1]\backslash [v]^{u(2z-j)} $, for $y \in [v]^{u(2z-j)}$ we get $ v(y) \geq  u(2z-j)$ so $u(2z-j) \wedge v(y)=u(2z-j)$. Analogously, if $y \in [0, 1] \backslash [v]^{u(2z-j)}$ we get $ v(y)< u(2z-j)$ so $u(2z) \wedge v(y)=v(y)$. Thus
$$\sup_{\frac{1}{2} x + \frac{j}{2} = z, \; y\in [0, 1]} \{ u(x) \wedge v(y) \}= \sup_{y\in [0, 1]} \{ u(2z -j) \wedge v(y)\}= u(2z -j).$$
For $j=0$
$$\phi_{0}(u \times v)(z)= \left\{
                                  \begin{array}{ll}
                                   u(2z), &\text{ if }  z \in [0, 1/2]\\
                                    0,  & \text{ if }  z \in (1/2,1]
                                  \end{array}
                                \right.
$$
For $j=1$
$$\phi_{1}(u \times v)(z)= \left\{
                                  \begin{array}{ll}
                                    0,  & \text{ if }  z \in [0,1/2)\\
                                    u(2z-1), &\text{ if }  z \in [1/2,1]
                                  \end{array}
                                \right.
$$

We  point out that $u(2 (1/2)  ) \vee u(2 (1/2) -1)=u(1) \vee u(0)$, but $0=1$ in $\mathbb{S}^1$.

From $$\mathcal{Z_S}(u,v)= \bigvee_{j=0,1} \phi_{j}(u \times v)$$ we get
$$\mathcal{Z_S}(u,v) (z)= u(T(z)) ,\; \; T(z):=2z \mod 1.$$

Obviously $\mathcal{Z_S}(1,1)=1$. We claim that this is the unique fixed point of $\mathcal{Z}_\mathcal{S}$ in $\mathcal{F}^*_X$, i.e.,
$\mathcal{Z_S}(u,u) (z)=u(z)$. Indeed, let $u\in \mathcal{F}_{X}^{*}$ be such that  $u(z)=u(T(z)),\; \forall  z \in [0, 1]$. Since $u$ is normal, there exists a point $a \in [0, 1]$ such that $u(a)= 1$. We define $\Gamma:=\bigcup_{n \in \mathbb{N}} T^{-n}(a)$. Obviously, $u(T^{-n}(a))=u(a), \; \forall n$. We claim that $\Gamma$ is dense in $[0,1]$. To see this, we take any $x \in [0,1]$ and let $x_{n}= d_{n-1} 2^{-1} + d_{n-2} 2^{-2} + d_{n-3} 2^{-3} + ....+d_{0} 2^{-n+1}, \; d_j \in \{0,1\}$, be a base $2$ truncated expansion of $x$. It is easy to see that $$y_{n}=x_{n} + a 2^{-n} = d_{n-1} 2^{-1} + d_{n-2} 2^{-2} + d_{n-3} 2^{-3} + ....+d_{0} 2^{-n+1} + a 2^{-n} \in \Gamma,$$ that is $T^{n}(y_{n})=a$, so $y_n \to x$. From the upper semicontinuity  of $u$ we get $\displaystyle u(x) \geq \limsup_{n \to \infty} u(y_{n}) = u(a)$. Thus $\displaystyle 1=u(a)= \min_{x \in [0,1]} u(x) \leq 1$. From this, we get $u\equiv 1$.
\end{example}
\begin{remark}\emph{
Actually, in any similar case, that is, $\mathcal{Z_S}=(X, (\phi_j)_{j=0,1}, (\rho_{j})_{j=0,1}),$  with $\rho_{j}(t)=t, \; t \in [0,1]$, $\bigcup_{j=0,1}\phi_j(X\times X)=X$ and $X$ compact, we get
$$\mathcal{Z_S}(1,1)(z)=\bigvee_{j=0,1} \phi_{j}(1 \times 1)(z)=\bigvee_{j=0,1} \sup_{(x,y) \in \phi_{j}^{-1}(z)}(1(x) \wedge 1(y)) =1,$$
because the property $\displaystyle\bigcup_{j=0,1}\phi_j(X\times X)=X$ implies that $\phi_{j}^{-1}(z)$ is never an empty set.
}\end{remark}

Now we prove that the operator $\mathcal{Z_S}$ satisfies the same contractive conditions as mappings from $\mathcal{S}$. Especially, this extends the first parts of Theorems \ref{cabrelli fixed point fractal operator} and \ref{GHB fuzzy attract exist and uniq}.

\begin{theorem} \label{Fractal operator is contractive}
Let $\mathcal{Z_S}=(X,(\phi_j)_{j=0,...,n-1},(\rho_j)_{j=0,...,n-1})$ be a Matkowski contractive GIFZS of degree $m$. Then the (FGHB) operator $\mathcal{Z_S}: (\mathcal{F}_{X}^*)^m \to \mathcal{F}_{X}^*$ is a generalized Matkowski contraction. Its witness is the function $\varphi:=\max_{j=0,...,n-1}\varphi_j$ where $\varphi_j$, $j=0,...,n-1$, are witnesses for $\phi_j,\;j=1,...,n-1$, respectively.\\
In particular, if $\mathcal{Z_S}$ is Lipschitz contractive, then the (FGHB) operator $\mathcal{Z_S}$ is also Lipschitz contractive with $Lip(\mathcal{Z_S})\leq\max_{j=0,...,n-1}Lip(\phi_j)<1$.
\end{theorem}
\begin{proof} Let $(u)_i,(v)_i\in(\mathcal{F}^*_X)^m$. We have
\begin{align*}
  d_\infty (\mathcal{Z_S}((u)_{i}),\mathcal{Z_S}((v)_{i})) & = \sup_{\alpha \in [0,1]} h([\mathcal{Z_S}((u)_{i})]^{\alpha}, [\mathcal{Z_S}((v)_{i})]^{\alpha}) \\
    & = \sup_{\alpha \in [0,1]} h\left(\bigcup_{j}\phi_{j}([\rho_{j}(\times_{i}  u_{i})]^{\alpha}), \bigcup_{j}\phi_{j}([\rho_{j}(\times_{i}  v_{i})]^{\alpha})\right) \\
\end{align*}
Taking $H_j=\phi_{j}([\rho_{j}(\btimes_{i}  u_{i})]^{\alpha})$ and $K_j=\phi_{j}([\rho_{j}(\btimes_{i}  v_{i})]^{\alpha})$ and applying Propositions~\ref{Mihail Dist  Unions} and \ref{filip1} we get
\begin{align*}
    & \leq \sup_{\alpha \in [0,1]} h\left(\bigcup_{j}\phi_{j}([\rho_{j}(\btimes_{i}  u_{i})]^{\alpha}), \bigcup_{j}\phi_{j}([\rho_{j}(\btimes_{i}  v_{i})]^{\alpha})\right) \\
& \leq \sup_{\alpha \in [0,1]} \sup_{j} h\left(\phi_{j}([\rho_{j}(\btimes_{i}  u_{i})]^{\alpha}), \phi_{j}([\rho_{j}(\btimes_{i}  v_{i})]^{\alpha})\right) \\
& \leq\sup_{\alpha\in[0,1]}\sup_j\varphi_j\left(h\left([\rho_{j}(\btimes_{i}  u_{i})]^{\alpha}, [\rho_{j}(\btimes_{i}  v_{i})]^{\alpha}\right)\right)\\
&\leq\sup_j\varphi_j\left(\sup_{\alpha \in [0,1]} h\left([\rho_{j}(\btimes_{i}  u_{i})]^{\alpha}, [\rho_{j}(\btimes_{i}  v_{i})]^{\alpha}\right)\right)\\
&=\varphi\left(\sup_{\alpha \in [0,1]} h\left([\rho_{j}(\btimes_{i}  u_{i})]^{\alpha}, [\rho_{j}(\btimes_{i}  v_{i})]^{\alpha}\right)\right)\\ &=\varphi(d_\infty^m((u)_i,(v)_i))
\end{align*}

Also, it is easy to see that $\varphi$ is nondecreasing and for every $t>0$, $\varphi^{(k)}(t)\to 0$.
\end{proof}

From Theorem~\ref{Fractal operator is contractive} and Theorem~\ref{Generalized Fixed Point Theorem} we had proved the main result, which gathers Theorems \ref{cabrelli fixed point fractal operator} and \ref{GHB fuzzy attract exist and uniq}.
\begin{theorem} \label{Existence and convergence fuzzy attrac GIFS}
If $X$ is complete, then there is a unique generalized fuzzy attractor $u_{\mathcal{Z}}\in \mathcal{F}_{X}^{*}$ for a Matkowski contractive GIFZS $\mathcal{Z_S}=(X, (\phi_j), (\rho_{j}))$, i.e., a unique $u_{\mathcal{Z}}\in\mathcal{F}^*_X$ such that $$u_\mathcal{Z}=\mathcal{Z_S}(u_\mathcal{Z},...,u_\mathcal{Z})= \bigvee_{j=0...n-1} \rho_{j}(\phi_{j}(u_\mathcal{Z}\times ... \times u_\mathcal{Z})).$$
Moreover, for every $u_0, u_1, . . . , u_{m-1} \in \mathcal{F}_{X}^{*}$, the sequence $(u_k)$ defined by
$$u_{k+m}:=\mathcal{Z_S}(u_{k+ m-1}, u_{k+ m-2}, . . . , u_{k}),$$
for all $k \in \mathbb{N}$, is convergent to $u_\mathcal{Z}$.
\end{theorem}

\begin{theorem}``GIFZS Collage Theorem'' \label{GIFS Collage Theorem} If $X$ is complete and $\mathcal{Z}_\mathcal{S}$ is a Lipschitz contractive GIFZS with $\lambda:=Lip(\mathcal{Z}_\mathcal{S})<1$, then for any $u \in \mathcal{F}_{X}^{*}$, we have
$$d(u, u_\mathcal{Z}) \leq \frac{1}{1-\lambda} \;d(u, \mathcal{Z_S}(u,...,u))$$
where $u_\mathcal{Z}$ is a unique fractal generated by $\mathcal{Z_S}$.
\end{theorem}
\begin{proof}
  The proof follow from Theorem~\ref{Collage Theorem}, taking $A=X$ and $T(u):=\mathcal{Z_S}(u,...,u)$. Theorem~\ref{Fractal operator is contractive} claims that
$$d_\infty(\mathcal{Z_S}((u)_{i}),\mathcal{Z_S}((v)_{i})) \leq \lambda \; d_\infty^{m}((u)_{i},(v)_{i}), \; \forall u_i,v_i \in \mathcal{F}_{X}^*,$$
with contraction constant $0\leq  \mathrm{Lip}(\phi_j) <1$. If $(u)_{i}=(u, ..., u)$ and $(v)_{i}=(v, ..., v)$ then $d_\infty^{m}((u)_{i},(v)_{i})= \max_{i}  d_\infty(u,v)= d_\infty(u,v)$. Thus, the above inequality became
$$d_\infty(T(u),T(v)) \leq \lambda \; d_\infty (u,v), \; \forall u ,v  \in \mathcal{F}_{X}^*.$$
From Theorem~\ref{Collage Theorem} for $T$ we get $$d(u, u_\mathcal{Z}) \leq \frac{1}{1-\lambda} \;d(u, \mathcal{Z_S}(u,...,u))$$
\end{proof}

\section{Further properties}

\subsection{Monotonicity of the generalized fuzzy operator}
To generalize the Proposition~\ref{Monotone Attractor and Fuzzy Attract} and to prove the analogous of Theorem~\ref{zero cut IFZS versus IFS}  in the GIFZS case we need to extend the notion of monotonicity.
\begin{definition}\emph{
Given $u_{0}, u_1,u_2,...,...  \in \mathcal{F}_{X}$, we say that this sequence is }nonincreasing\emph{, if $u_{0} \geq u_{1} \geq...$, and }nondecreasing, \emph{if $u_{0} \leq u_{1} \leq...$.}
\end{definition}
\begin{lemma}\label{GIFZS preserves monotonous sequences}
  The operator  $\mathcal{Z_S}$, associated to a GIFZS, preserves monotonous sequences in $\mathcal{F}_{X}^*$. More  precisely, if  $u_{0} \geq \cdots \geq u_{m-1}$, the sequence $u_{k+m}=\mathcal{Z_S}(u_{k+ m-1}, u_{k+ m-2}, . . . , u_{k}),$
for all $k \in \mathbb{N}$, is nonincreasing in  $\mathcal{F}_{X}^*$ provided $u_{m-1}\geq u_m$. The same is true for a nondecreasing sequence.
\end{lemma}
\begin{proof}  By assumption, we have that $u_0\geq u_1\geq...\geq u_{m-1}\geq u_m$.
  Suppose that $u_{k} \geq \cdots \geq u_{k+m-1}$ for some $k\geq 1$. We will prove that $u_{k+m-1} \geq u_{k+m}$. Using the fact that $u_{k+m}=\mathcal{Z_S}(u_{k+ m-1}, u_{k+ m-2}, . . . , u_{k})$ and $u_{k+m-1}=\mathcal{Z_S}(u_{k+ m-2}, u_{k+ m-3}, . . . , u_{k-1}),$ we need to compare 
$$\phi_j(u_{k+ m-1} \times . . . \times u_{k})(z) \text{ and } \phi_j(u_{k+ m-2} \times  . . . \times u_{k-1})(z), \forall z \in X$$

If $(\phi_j)^{-1}(z)=\varnothing$ then $\phi_j(u_{k+ m-1} \times  . . . \times u_{k})(z)=0  \leq \phi_j(u_{k+ m-2} \times . . . \times u_{k-1})(z)$. Otherwise, if $(x_0, ..., x_{m-1}) \in (\phi_j)^{-1}(z) \neq \varnothing$ then $$u_{k+ m-1}(x_0) \leq u_{k+ m-2}(x_0),   \cdots, u_{k}(x_{m-1}) \leq u_{k-1}(x_{m-1})$$
thus $\phi_j(u_{k+ m-1} \times  . . . \times u_{k})(z)\leq \phi_j(u_{k+ m-2} \times  . . . \times u_{k-1})(z)$.
Hence we get
$$\mathcal{Z_S}(u_{k+ m-1}, u_{k+ m-2}, . . . , u_{k})= \bigvee_{j=0...n-1} \rho_j(\phi_j(u_{k+ m-1} \times . . . \times u_{k})) \leq $$ $$\leq \bigvee_{j=0...n-1} \rho_j(\phi_j(u_{k+ m-2} \times  . . . \times u_{k-1})) = \mathcal{Z_S}(u_{k+ m-2}, u_{k+ m-3}, . . . , u_{k-1}),$$
because the grey scale maps are nondecreasing.
\end{proof}

\begin{proposition} \label{Monotone Attractor and Fuzzy Attract GIFZS} Let the crisp set $A_\mathcal{S} \in \mathcal{K}^*(X)$ be the attractor of a Matkowski contractive GIFS $\mathcal{S}=(X, (\phi_j)_{j=0...n-1})$, and $u_\mathcal{Z} \in \mathcal{F}_{X}^*$ be the fuzzy attractor of the GIFZS $\mathcal{Z_S}=(X, (\phi_j)_{j=0...n-1}, (\rho_j)_{j=0...n-1})$. Then, for any  $B \in \mathcal{K}^*(X)$ and $v \in \mathcal{F}_{X}^*$ we have\\
a) If $\mathcal{Z_S}(v \times \cdots \times v) \leq v$ then $u_\mathcal{Z} \leq v$;\\
b) If $\mathcal{S}(B \times \cdots \times B) \subseteq B$ then $A_\mathcal{S} \subseteq B$;\\
c) If $v \leq \mathcal{Z_S}(v \times \cdots \times v)$ then $v \leq u_\mathcal{Z}$;\\
d) If $B \subseteq \mathcal{S}(B \times \cdots \times B) $ then $B \subseteq A_\mathcal{S}$.
\end{proposition}
\begin{proof}
a) If $\mathcal{Z_S}(v \times \cdots \times v) \leq v$ we get, from Lemma \ref{GIFZS preserves monotonous sequences}, a nonincreasing sequence $u_{0}=v \geq \cdots \geq u_{m-1}=v $ and $u_{k+m}=\mathcal{Z_S}(u_{k+ m-1}, u_{k+ m-2}, . . . , u_{k})$
for all $k \in \mathbb{N}$.  From  Theorem~\ref{Existence and convergence fuzzy attrac GIFS} we have $\lim_{k \to \infty} u_{k}=u_\mathcal{Z}$ so $u_\mathcal{Z} \leq v$. The other items are proved in the same fashion.
\end{proof}

\subsection{Relationships between a GIFZS and the appropriate GIFS}


Here we will investigate the basic relationships between GIFZS $\mathcal{Z}_\mathcal{S}$ and the GIFS $\mathcal{S}$. As iterated fuzzy function systems are particular versions of GIFZSs, the results are also true for them.

\begin{theorem} \label{zero cut GIFZS versus GIFS} Assume that $X$ is complete and $\mathcal{Z_S}=(X, (\phi_j)_{j=0...n-1}, (\rho_j)_{j=0...n-1})$ is a Matkowski contractive GIFZS with the attractor $u_\mathcal{Z}$, and let $$I:=\{j\in\{0,..,n-1\}:\rho_j(1)=1\}\}$$ and $\mathcal{S'}:=(X,(\phi_j)_{j\in I})$. Then we have\\
\vspace{.3cm}
(1) $[u_\mathcal{Z}]^{0} \subseteq A_\mathcal{S}$, and if $r^+_j=0$ for all $j=0,...,n-1$, then $[u_\mathcal{Z}]^0=A_\mathcal{S}$;\\
\vspace{.3cm}
(2) $A_\mathcal{S'}\subseteq [u_\mathcal{Z}]^1$, and if $\beta_j(1)=1$ for all $j\in I$, then $A_\mathcal{S'}= [u_\mathcal{Z}]^1$.\\
\vspace{.3cm}
(3) If $I=\{0,...,n-1\}$, then $u_\mathcal{Z}=\chi_{A_\mathcal{S}}$.
\end{theorem}
\begin{proof} Ad(1) Consider $\bchi_{A_\mathcal{S}} \in \mathcal{F}_{X}^*$ then,
$$ \phi_{j}(\btimes_i \bchi_{A_\mathcal{S}})(z)=\sup_{\phi_{j}(x_0, ...,x_{m-1})=z} \bigwedge_{i} \bchi_{A_\mathcal{S}}(x_i) ,$$
if $(\phi_j)^{-1}(z)=\varnothing$ the inequality is trivial.
We notice that
$$\bigwedge_{i} \bchi_{A_\mathcal{S}}(x_i) = \bchi_{A_\mathcal{S}\times \cdots \times A_\mathcal{S}} (x_0, ... , x_{m-1})$$ thus
$$ \phi_{j}(\btimes_i \bchi_{A_\mathcal{S}})(z)=\sup_{\phi_{j}(x_0, ...,x_{m-1})=z}  \bchi_{A_\mathcal{S}\times \cdots \times A_\mathcal{S}} (x_0, ... , x_{m-1})=\bchi_{\phi_{j}(A_\mathcal{S}\times \cdots \times A_\mathcal{S})} (z) \leq \bchi_{A_\mathcal{S}}(z),$$
from Remark~\ref{Crisp Set Extension to Fuzzy Set} and because $\displaystyle A_\mathcal{S} = \bigcup_{j=0}^{n-1} \phi_{j}(A_\mathcal{S}\times \cdots \times A_\mathcal{S})$ implies $\phi_{j}(A_\mathcal{S}\times \cdots \times A_\mathcal{S}) \subseteq A_\mathcal{S}$ for any $j$.

In particular $\rho_j(\phi_{j}(\btimes_i \bchi_{A_\mathcal{S}})) \leq  \rho_j( \bchi_{A_\mathcal{S}}) \leq \bchi_{A_\mathcal{S}}$ for any $j$, because each $\phi_{j}$ is not decreasing and $\rho_j(0)=0$ and $\rho_j(1) \leq 1$.  Since $\mathcal{Z_S}$ is a supremum we get $\mathcal{Z_S}(\bchi_{A_\mathcal{S}},...,\bchi_{A_\mathcal{S}}) \leq \bchi_{A_\mathcal{S}}$. From Proposition~\ref{Monotone Attractor and Fuzzy Attract GIFZS} we get $u_\mathcal{Z} \leq \bchi_{A_\mathcal{S}}$ thus $[u_\mathcal{Z}]^{0} \subseteq A_\mathcal{S}$, because if $x \not\in A_\mathcal{S}$ then $  \bchi_{A_\mathcal{S}}(x)=0 \geq u_\mathcal{Z}(x) \geq 0$ so $u_\mathcal{Z}(x) =0$.\\
To prove the second part, assume that $r_+^j=0$ for all $j=0,...,n-1$. Then by Remark \ref{f3'} (note that here $\rho_j(r_+^j)=\rho_j(0)=0$), we have
$$
[u_\mathcal{Z}]^0=[\mathcal{Z_S}(\times_iu_\mathcal{Z})]^0=\bigcup_{j=0,...,n-1}\phi_j([\rho_j(\times_iu_\mathcal{Z})]^0)=\bigcup_{j=0,...,n-1}\phi_j([u_\mathcal{Z}]^0\times...\times[u_\mathcal{Z}]^0)
$$
Hence $[u_\mathcal{Z}]^0=A_\mathcal{S}$ by the uniqueness of the attractor of a GIFS.\\
Ad (2) By Remark \ref{f3'}
$$
[u_\mathcal{Z}]^1=\bigcup_{j=0,...,n-1}\phi_j([\rho_j(\times_iu_i)]^1)=\bigcup_{j\in I}\phi_j([\rho_j(\times_iu_i)]^1)=$$ \begin{equation}\label{f4'}=
\bigcup_{j\in I}\phi_j(\times_i[u_\mathcal{Z}]^{\beta_j(1)})\supset \bigcup_{j\in I}\phi_j(\times_i[u_\mathcal{Z}]^{1})=\mathcal{S'}([u_\mathcal{Z}]^1,...,[u_\mathcal{Z}]^1)
\end{equation}
where the second equality follows from the fact that if $\rho_j(1)<1$, then $[\rho_j(\times u_\mathcal{Z})]^1=\emptyset$ (Proposition \ref{alpha cut compose with ndrc rho}(b)).
Hence by Proposition \ref{Monotone Attractor and Fuzzy Attract GIFZS} we get $[u_\mathcal{Z}]^1\supset A_\mathcal{S'}$. Also, if for all $j\in I$, $\beta_j(1)=1$, then in (\ref{f4'}) we have all equalities, so in this case $[u_\mathcal{Z}]^1=A_\mathcal{S'}$.\\
Ad (3) If $I=\{0,...,n-1\}$, then by (1) and (2) we have
$$
A_\mathcal{S}=A_\mathcal{S'}\subset [u_\mathcal{Z}]^1\subset [u_\mathcal{Z}]^0\subset A_\mathcal{S}
$$
so $[u_\mathcal{Z}]^1= [u_\mathcal{Z}]^0=A_\mathcal{S}$, and this implies $u_\mathcal{Z}=\chi_{A_\mathcal{S}}$.
\end{proof}

The above result suggests the definition:
\begin{definition}\emph{
An admissible system of grey level maps $(\rho_j)_{j=0,...,n-1}$ is called }proper\emph{, if $r^+_j=0$ and $\beta_j(\rho_j(1))=1$ for all $j=0,...,n-1$.}
\end{definition}
As a corollary of the above proof, we have that
\begin{corollary}\label{f3}
Let $\mathcal{Z_S}=(X,(\phi_j)_{j=0,...,n-1},(\rho_j)_{j=0,...,n-1})$ be a Matkowski contractive GIFZS on a complete metric space with a proper family $(\rho_j)$. If $I:=\{j:\rho_j(1)=1\}$, $\mathcal{S}:=(X,(\phi_j))$ and $\mathcal{S}':=(X,(\phi_j)_{j\in I})$, then $[u_\mathcal{Z}]^0=A_\mathcal{S}$ and $[u_\mathcal{Z}]^1=A(\mathcal{S}')$, where $u_\mathcal{Z}$ is the attractor of $\mathcal{Z_S}$.
\end{corollary}

\begin{remark}\label{f4}\emph{
The above result gives a natural sufficient condition under which the attractor of a GIFZS is not a crisp set. Indeed, take a GIFS $\mathcal{S}=(X,(\phi_j)_{j=0,...,n-1})$ whose attractor is not a singleton and let $I\subset\{0,...,n-1\}$ be such that the attractor of a GIFS $\mathcal{S}'=(X,(\phi_j)_{j\in I})$ does not equal $A_\mathcal{S}$ (for example, $I$ can be a singleton). Finally take a proper family $(\rho_j)$ of grey level maps so that $\rho_j(1)=1$ if $j\in I$ and $\rho_j(1)<1$ if $j\notin I$. Then by Corollary \ref{f3}, the attractor $u_\mathcal{Z}$ of a GIFZS $\mathcal{Z_S}=(X,(\phi_j),(\rho_j))$ is not a crisp set.}
\end{remark}

\begin{example} In this example, we explore the symmetries of a fixed GIFS to give some general properties of the fuzzy attractor for any admissible family of grey level maps. Consider $\mathcal{Z_S}=([0,1], (\phi_j)_{j=0,1}, (\rho_{j})_{j=0,1}),$  where $\rho_{j}(t), \; t \in [0,1]$ is an arbitrary family of admissible grey level maps and the Lipschitz maps are given by $\phi_j(x,y):=\frac{1}{4} x + \frac{1}{4} y +\frac{j}{2}, \; j=0,1$. In this case,
$$\phi_{j}(u \times v)(z)= \left\{
                                  \begin{array}{ll}
                                   \sup_{x,y \in [0,1], \;  y = 4z -2j -x} u(x) \wedge v(y), &\text{ if } \exists x,y \;s.t. \;y = 4z -2j -x\\
                                    0,  & \text{otherwise}
                                  \end{array}
                                \right.
$$
First we need to compute $\phi_j^{-1}(z)$. It is easy to see that
$$\phi_0^{-1}(z)=
\left\{
  \begin{array}{ll}
     \{(x, 4z -x), x \in [0,4z]\}, & z \in [0, 1/4]; \\
     \{(x, 4z -x), x \in [4z-1, 1]\}, & z \in [1/4, 1/2]; \\
 \varnothing, & z\in (1/2, 1].
  \end{array}
\right.
$$
so
$$\phi_0(u,u)(z)=
\left\{
  \begin{array}{ll}
     \sup_{x \in [0,4z]} u(x)\wedge u(4z -x), & z \in [0, 1/4]; \\
     \sup_{x \in [4z-1,1]} u(x)\wedge u(4z -x), & z \in [1/4, 1/2]; \\
 0, & z\in (1/2, 1].
  \end{array}
\right.
$$
We can also prove that
$$\phi_1^{-1}(z)=
\left\{
  \begin{array}{ll}
     \{\varnothing\}, & z\in [0, 1/2)\\
    \phi_0^{-1}(z-1/2), & z \in [1/2, 1].
  \end{array}
\right.
$$
so
$$\mathcal{Z_S}(u,u)(z)=\bigvee_{j=0,1} \rho_{j}(\phi_{j}(u\times u)(z))=
\left\{
  \begin{array}{ll}
    \rho_{0}(\phi_{0}(u\times u)(z)), & z \in [0, 1/2); \\
     \rho_{0}(u(1)) \vee \rho_{1}(u(0)), & z= 1/2;\\
     \rho_{1}(\phi_{0}(u\times u)(z-1/2)), & z \in (1/2, 1]. \\
  \end{array}
\right.
$$
Suppose that $u_\mathcal{Z}$ is the fuzzy fractal of $\mathcal{Z}_\mathcal{S}$. Then $$u_\mathcal{Z}(z)= \rho_{0}(\phi_{0}(u_\mathcal{Z}\times u_\mathcal{Z})(z)), \; z \in [0, 1/2)$$
and $$u_\mathcal{Z}(z)= \rho_{1}(\phi_{0}(u_\mathcal{Z}\times u_\mathcal{Z})(z-\frac{1}{2})), \; z \in (1/2, 1]$$

A particular case is $\rho_{0}(t):=\frac{1}{2}t$ and $\rho_{1}(t):=t$. In this situation $$u_\mathcal{Z}(z)= \frac{1}{2 }\phi_{0}(u_\mathcal{Z}\times u_\mathcal{Z})(z), \; z \in [0, 1/2)$$
and $$u_\mathcal{Z}(z)= \phi_{0}(u_\mathcal{Z}\times u_\mathcal{Z})(z-1/2), \; z \in (1/2, 1]$$
Now let $z \in (1/2, 1)$. Then $z-1/2 \in (0, 1/2)$, so we have
$$u_\mathcal{Z}(z)= \phi_{0}(u_\mathcal{Z}\times u_\mathcal{Z})(z-1/2)= 2 \left( \frac{1}{2 }\phi_{0}(u_\mathcal{Z}\times u_\mathcal{Z})(z-1/2)\right)=2 u_\mathcal{Z}(z-1/2)$$

Another particular case is $\rho_{0}(t):=t$ and $\rho_{1}(t):=0$. In this situation the attractor $u_\mathcal{Z}$ is given by $u_\mathcal{Z}(z)=\bchi_{\{0\}}(z)$.
\end{example}

\subsection{Richness of the class of generalized fuzzy fractals}

Now we are going to investigate the class of generalized fuzzy fractals. In particular, we give a partially positive answer to the question whether such class is essentially wider than the class of fuzzy fractals generated by IFZSs. We also show that for certain complete metric spaces, the family of IFZSs' fractals is dense in $\mathcal{F}^*_X$ (this result is in fact a slight generalization of already known result due to Cabrelli et al.).


Let $$\mathcal{A}_i:=\{ u \in \mathcal{F}_{X}^* \;|\;  \mathcal{Z_S}( u)=u\text{ for some Lipschitz contractive IFZS}\}$$
$$\mathcal{M}_i:=\{ u \in \mathcal{F}_{X}^* \;|\;  \mathcal{Z_S}( u)=u\text{ for some Matkowski contractive IFZS}\}$$
 and if $m\in\N$, then
$$\mathcal{A}_g^m:=\{ u \in \mathcal{F}_{X}^* \;|\;  \mathcal{Z_S}( u, ...., u)=u\text{ for some Lipschitz contractive GIFZS of degree } m\} \subseteq \mathcal{F}_{X}^*.$$
and finally
$$\mathcal{M}_g^m:=\{ u \in \mathcal{F}_{X}^* \;|\;  \mathcal{Z_S}( u, ...., u)=u\text{ for some Matkowski contractive GIFZS of degree } m\} \subseteq \mathcal{F}_{X}^*.$$

Clearly, $\mathcal{A}_i=\mathcal{A}_g^1$  and $\mathcal{M}_i=\mathcal{M}_g^1$.

The next result is analogous to Strobin~\cite{MR3263451} and Miculesu \cite{MR3180942} (see also Miculescu and Mihail \cite{MR2415407}) in the fuzzy setting.

\begin{theorem} \label{the set of attrac of GIFZS is bigger than IFZS}$\;$\\
(1) For every $m\in\N$, $\displaystyle\mathcal{A}^m_g \subset  \displaystyle\mathcal{A}^{m+1}_g$ and $\displaystyle\mathcal{M}^m_g \subset  \displaystyle\mathcal{M}^{m+1}_g$.\\
(2) There exists a complete metric space $X$ such that $\mathcal{A}_i\subsetneq \mathcal{A}^2_g$.\\
(3) Let $X=\R^2$ and $m>1$. Then there exists $u^*\in\mathcal{F}_X^*$ such that:\\
- $u^*$ is the attractor of some Lipschitz contractive GIFZS $\mathcal{Z_S}=(X, (\phi_j), (\rho_j))$ of degree $m$ with proper system $(\rho_j)$;\\
-  $u^*$ is not the attractor of any Matkowski contractive GIFZS $\mathcal{Z_S}=(X, (\phi_j), (\rho_j))$ of degree $m-1$ consisting of an admissible system $(\rho_j)$ with $r_j^+=0$ for all $j$;\\
- $u^*$ is not the attractor of any GIFS.
\end{theorem}
\begin{proof}
Ad(1) We want to prove that $\mathcal{M}_g^{m-1} \subseteq \mathcal{M}_g^{m}$ for all $m > 1$.
Suppose $v \in \mathcal{M}^{m-1}_g$, that is  $\mathcal{Z_R}(v)= \bigvee_{j=0...n-1} \rho_j(\psi_j(\times_iv))=v,$  for a Matkowski contractive GIFZS $\mathcal{Z_{S}}=(X, (\psi_j)_{j=0...n-1}, (\rho_j)_{j=0...n-1})$. Let us consider the degree $m$ GIFZS $$\mathcal{Z_{S'}}=(X, (\phi_j)_{j=0...n-1}, (\rho_j)_{j=0...n-1})$$ 
where for every $u_{0}, ..., u_{m-1}  \in \mathcal{F}_{X}^{*}$, $$\phi_{j}(u_{0}, ..., u_{m-1})= \psi_{j}( u_{0},...,u_{m-2}) $$
Obviously, $\phi_j$ are generalized Matkowski contractions.\\
We recall that for $j=0,...,n-1$ and $u_{0}, ..., u_{m-1} \in \mathcal{F}_{X}^*$, we have that $$\phi_{j}(\btimes_{i}  u_{i})(z):= \left\{
                                  \begin{array}{ll}
                                   \sup_{\phi_{j}((x)_{i})=z} \bigwedge_{i=0}^{m-1 }u_{i}(x_{i}), &\text{ if } \phi_{j}^{-1}(z) \neq \varnothing\\
                                    0,  & \text{otherwise}
                                  \end{array}
                                \right.
  .$$

But $\phi_{j}^{-1}(z)=\{ (x_0, ... ,x_{m-1}) \; |\; \psi_{j}(x_0,...,x_{m-2})=z\}=\psi_{j}^{-1}(z)\times X$.

On the other hand
$$\psi_{j}( u_{0}\times...\times u_{m-2})(z)= \left\{
                                  \begin{array}{ll}
                                   \sup_{\psi_{j}(x_{0},...,x_{m-2})=z} \bigwedge_{i=0}^{m-2}u_{i}(x_{i}), &\text{ if } \psi_{j}^{-1}(z) \neq \varnothing\\
                                    0,  & \text{otherwise.}
                                  \end{array}
                                \right.
$$

We claim that  $$\phi_{j}(\times_{i=0}^{m-1}v)(z)= \psi_{j}( \times_{i=0}^{m-2}v)(z).$$

Indeed, if $ (x_0, ...., x_{m-1}) \in \phi_{j}^{-1}(z)=\{ (x_0, ... ,x_{m-1}) \; |\; \psi_{j}(x_0,...,x_{m-2})=z\}=\psi_{j}^{-1}(z)\times X \neq \varnothing$ then
$$\bigwedge_{i=0}^{m-1 }v(x_{i})=
\left\{
  \begin{array}{ll}
    \bigwedge_{i=0}^{m-2}v(x_i), & if v(x_{m-1}) \geq v(x_{i}), \;\mbox{for some}\; i\leq m-2 \\
    v(x_{m-1}), & if v(x_{i}) \geq v(x_{m-1}),\;\mbox{for all}\; i\leq m-2\\
  \end{array}
\right.
$$
so $$\phi_{j}(\btimes_{i}  v)(z) =\sup_{\phi_{j}(x_0,...,x_{m-1})=z} \bigwedge_{i=0}^{m-1 }v(x_{i})= \sup_{\psi_{j}(x_{0},...,x_{m-2})=z} \bigwedge_{i=0}^{m-2} v(x_{i})=\psi_{j}(\btimes_{i} v)(z).$$
Thus $\mathcal{Z_S}(v,...,v)=\mathcal{Z_{S'}}(v,...,v)= v \in \mathcal{M}_g^m$.\\
In the same way we show the first inclusion in (1).\\
Ad(2) To show the inequality we are going to show that a certain GIFZS has an attractor that is not the attractor of any IFZS.

In Miculescu~\cite{MR3180942} Remark 4.7, we find the following example of GIFS. Let $X:=\prod_{r=0}^{\infty}[0, \frac{1}{4^r}]$, and consider it as a compact space with the product metric. Let us point out that the Hausdorff dimension of $X$ is infinite.

Consider the GIFZS $\mathcal{Z_S}=(X, (\phi_j)_{j=0,1}, (\rho_{j})_{j=0,1}),$  with $\rho_{0}(t):=t, \rho_{1}(t):= t, \; t \in [0,1]$ and $\phi_j(x,y):=(\frac{1}{2} x_{0} + \frac{j}{2}, \frac{1}{4}y), \; j=0,1$. 
Let us to consider the operator
$$\mathcal{Z_S}(u,v)= \bigvee_{j=0,1} \rho_{j}(\phi_{j}(u \times v))= \bigvee_{j=0,1} \phi_{j}(u \times v).$$

It is easy to see that $\mathcal{Z_S}(1,1)=1$ because $\phi_0(X,X) \cup \phi_1(X,X) =X$ and, under the hypothesis of Theorem~\ref{zero cut GIFZS versus GIFS}(3) we know that $u^* =1=\bchi_X$ is the unique fuzzy attractor of this GIFZS.
Suppose that $u^*  \in \mathcal{A}_i$. Then Theorem~\ref{zero cut IFZS versus IFS} implies that $[u^*]^0 \subseteq A_\mathcal{R}$, the attractor of the associated Lipschitz contractive IFS $\mathcal{R}$. In particular $[u^*]^0$ has finite Hausdorff dimension, contradicting $[u^*]^0= X$.\\
Ad(3) Let $m>1$. By \cite{MR3263451} we know that there is a Lipschitz contractive GIFS $\mathcal{S}=(\R^2,(\phi_j)_{j=0,...,3})$ of degree $m$ whose attractor $A_\mathcal{S}$ is a Cantor-type subset of $\R^2$ which is not an attractor of any GIFS of degree $m-1$. Let $\rho_j(t):=\frac{1}{2}t$ for $j=0,1,2$ and $\rho_3(t):=t$, and let $\mathcal{Z_S}:=(\R^2,(\phi_j)_{j=0,...,3},(\rho_j)_{j=0,..,3})$ and $u_\mathcal{Z}$ be its attractor. Since $(\rho_j)_{j=0,...,3}$ is proper and the attractor of a one-element GIFS $(\R^2,\phi_3)$ is a singleton, by Remark \ref{f4} we have that $u_\mathcal{Z}$ is not crisp and by Theorem \ref{zero cut GIFZS versus GIFS}, $[u_\mathcal{Z}]^0=A_\mathcal{S}$.\\
Now assume that $u_\mathcal{Z}$ is the attractor of some Matkowski contractive GIFZS $\mathcal{Z_{S'}}=(\R^2,(\psi_j),(\eta_j))$  of degree $m-1$ with the system $(\eta_j)$ satisfying $r_j^+=0$ for all $j$. Then by Theorem \ref{zero cut GIFZS versus GIFS}(1), we get that $A_\mathcal{S}=[u_\mathcal{Z}]^0=A_\mathcal{S'}$, which is a contradiction.
\end{proof}

\begin{remark}\emph{
The proof of part (1) of the above suggests a canonical way to built a GIFS $\mathcal{S}=(X, (f_j)_{j=0...n-1})$ with the same attractor of a given Matkowski contractive IFS $\mathcal{R}=(X, (\phi_j)_{j=0...n-1})$. We simply define
$$\phi_j(x_0,...,x_{n-1}):= f_j(x_0), \; j=0..n-1.$$
It is easy to see that $\mathcal{S}(A_\mathcal{R}, ..., A_\mathcal{R})=A_\mathcal{R}$, and thus from the uniqueness of $A_\mathcal{S}$ we obtain $A_\mathcal{R}=A_\mathcal{S}$.\\
The same is true for the fuzzyfied ones. Given an IFZS  $\mathcal{Z_R}=(X, (f_j)_{j=0...n-1}, (\rho_j)_{j=0...n-1})$ we define an $\mathcal{Z_S}=(X, (\phi_j)_{j=0...n-1}, (\rho_{j})_{j=0...n-1}),$ where
$$\phi_{j}(\btimes_{i}  u_{i}):=f_j(u_0)$$
We notice that
$$\phi_{j}(\btimes_{i}  u_{i})(z)=
\left\{
                                  \begin{array}{ll}
                                   \sup_{f_{j}(x_{0})=z} \bigwedge_{i=0}^{m-1 }u_{i}(x_{i}), &\text{ if } f_{j}^{-1}(z)\times X^{n-1} \neq \varnothing\\
                                    0,  & \text{otherwise}
                                  \end{array}
\right.
=$$ $$=
\left\{
                                  \begin{array}{ll}
                                   \sup_{f_{j}(x_{0})=z} u_{0}(x_{0}), &\text{ if } f_{j}^{-1}(z) \neq \varnothing\\
                                    0,  & \text{otherwise}
                                  \end{array}
\right.
= f_j(u_0)(z)$$
Again it is immediate to see that if  $u^*$ is the fuzzy attractor of $\mathcal{Z_R}$ then
$$\mathcal{Z_S}(u^*, ..., u^*)= \bigvee_{j=0...n-1} \rho_{j}(\phi_{j}(\btimes_{i}  u^*))=\bigvee_{j=0...n-1} \rho_{j}(f_{j}(u^*))=\mathcal{Z_R}(u^*)=u^*.$$
So $u^*$ is a fuzzy fractal also for $\mathcal{Z_S}$.}
\end{remark}

As an application of the GIFZS Collage Theorem~\ref{GIFS Collage Theorem} we get a density result, which is a slight extension of \cite{MR1187310}, Theorem 3.1.

\begin{theorem} \label{Density of gen fuzzy attract} Assume that $(X,d)$ is a complete metric space which has a Lipschitz contraction structure, i.e., which satisfies
$$\forall K\in\mathcal{K}(X)\;\forall L>0\;\forall x \in X \;\forall \varepsilon >0\; \exists \;f:X \to X,\; \mathrm{Lip}(f) <L,\; s.t. \;f(K) \subset B(x, \varepsilon),$$ (for example if $X=\R^n$). Then $\overline{\mathcal{A}_i}= \mathcal{F}_{X}^*$, where the closure is taken in the metric  $d_\infty$. In particular $\overline{\mathcal{A}_g^m}= \mathcal{F}_{X}^*$ for all $m \geq 2$.
\end{theorem}
\begin{proof}
We prove for IFZS first, and extend to  GIFZS after. The proof of this part is due to \cite{MR1187310} Theorem 3.1, slightly generalized for spaces with Lipschitz contraction structure. We present it here for completeness, but the reasoning is essentially the same.

Take $u \in \mathcal{F}_{X}^*$. Then $[u]^0 \subseteq X$ is a compact set. Hence for each $\varepsilon>0$ there exists $N \in \mathbb{N}$ and $x_0,...,x_{N-1}$ such that $\displaystyle \bigcup_{j=0...N-1}B(x_j, \varepsilon/4) \supset   [u]^0$. Clearly, we can assume that $u(x_j)>0$ for $j=0,...,N-1$.
Using the contraction structure we can find $(\phi_j)_{j=0...N-1}$ such that
$$\phi_j([u]^0) \subset B(x_j, \varepsilon/4), j=0...N-1,$$
and $\lambda_j:= \mathrm{Lip}(\phi_j) <\frac{1}{2}$, if it not happens we can consider a power of the initial contraction.

Chose now $\displaystyle \alpha_j:= \sup_{x \in \overline{B}(x_j, \varepsilon/4)} u(x)$. Then each $\alpha_j>0$ and $\displaystyle [u]^{\alpha} \subseteq \bigcup_{\alpha_j \leq \alpha} [u]^{\alpha_j}$ for any $\alpha\in[0,1]$. Indeed, it follows from $u(x)\geq \alpha \geq \alpha_j$.

Now define $\rho_j (x):= \alpha_j \bchi_{[\alpha_j , 1]}(x)$, that is obviously ndrc. In particular  $\rho_j (0)=0$,  $\rho_j (1)=\alpha_j$ and $\beta_j(\alpha)=\alpha_j$ for $0<\alpha\leq \alpha_j$. Consider the IFZS $\mathcal{Z_R}:=(X, (\phi_j)_{j=0...n-1}, (\rho_j)_{j=0...n-1})$.

By Proposition \ref{alpha cut compose with ndrc rho},
\[
[\rho_j(u)]^{\alpha}
=
  \begin{cases}
    \varnothing, & \alpha > \alpha_j \\
    [u]^{\alpha_{j}}, &  0 < \alpha \leq \alpha_j.
  \end{cases}
\]

By Lemma~\ref{general properties of fuzzy composition}
c), we get $\displaystyle \left[\mathcal{Z_R}(u)\right]^{\alpha}= \bigcup_{j=0...n-1}\phi_{j}([\rho_{j}(u)]^{\alpha})=  \bigcup_{j: \alpha \leq \alpha_j} \phi_{j}([u]^{\alpha_{j}})$ for $\alpha\in(0,1]$.\\
So fix $\alpha\in(0,1]$. Then by the above we have
\begin{equation}\label{f5}
h([u]^\alpha,[\mathcal{Z_R}(u)]^\alpha)= h\left([u]^\alpha,\bigcup_{j: \alpha \leq \alpha_j} \phi_{j}([u]^{\alpha_{j}})\right)
\end{equation}
Take any $x\in [u]^\alpha$. Then for some $j$, $x\in B(x_j,\frac{\varepsilon}{4})$. By definition, $\alpha_j\geq u(x)\geq \alpha$. Also, $\emptyset\neq\phi_j([u]^{\alpha_j})\subset \phi_j([u]^0)\subset B(x_j,\frac{\varepsilon}{4})$. This gives us
$$
[u]^\alpha\subset \left(\bigcup_{j:\alpha\leq\alpha_j}\phi_j([u]^{\alpha_j})\right)_{\frac{\varepsilon}{2}}
$$
Conversely, let $x\in \phi_j([u]^{\alpha_j})$ for some $j$ with $\alpha\leq\alpha_j$. Then $x\in B(x_j,\frac{\varepsilon}{4})$, and by definition of $\alpha_j$ and a compactness of the set $\overline{B}(x_j,\frac{\varepsilon}{4})\cap [u]^0$, there is $y\in \overline{B}(x_j,\frac{\varepsilon}{4})$ such that $u(y)=\alpha_j$ so $y\in[u]^\alpha$. In particular,
$$
\bigcup_{j:\alpha\leq\alpha_j}\phi_j([u]^{\alpha_j})\subset \left([u]^\alpha\right)_{\frac{\varepsilon}{2}}
$$
Hence by (\ref{f5}) and Corollary \ref{f2}, we have
$$
d_\infty(u, \mathcal{Z_R}(u))= \sup_{\alpha\in(0,1]} h([u]^{\alpha}, [\mathcal{Z_R}(u)]^{\alpha})\leq\frac{\varepsilon}{2}
$$

By the IFZS Collage Theorem we get $d_\infty(u, u^*) <\frac{1}{1-1/2} \varepsilon/2 = \varepsilon$ where $u^*$ is the attractor of $\mathcal{Z_R}$. Hence we get $\overline{\mathcal{A}_i}= \mathcal{F}_{X}^*$. By Theorem~\ref{the set of attrac of GIFZS is bigger than IFZS} we have $\displaystyle\mathcal{A}_i \subset\mathcal{A}_g^m$ for every $m\in\N$, so $\overline{\mathcal{A}_i}= \mathcal{F}_{X}^*$, implies that $\overline{\mathcal{A}_g^m}= \mathcal{F}_{X}^*$ for all $m \geq 2$.
\end{proof}

\addcontentsline{toc}{section}{References}

\end{document}